\documentclass[a4paper,notitlepage,twoside,reqno,11pt]{amsart}

\usepackage{anysize}
\marginsize{3.4cm}{3.4cm}{3cm}{3cm}

\usepackage{amsmath}
\usepackage{amsthm}
\usepackage{amssymb}
\usepackage{color}
\usepackage{indentfirst}
\usepackage{graphicx}
\usepackage{bbding}

\usepackage[hypertexnames=false,bookmarksopen=true,linktocpage=true,pdfstartview={XYZ null null 1.25}]{hyperref}

\usepackage{autonum}

\begin{document}
	
\newtheorem{thm}{Theorem}[section]
\newtheorem{lem}[thm]{Lemma}
\newtheorem{cor}[thm]{Corollary}
\newtheorem{prop}[thm]{Proposition}
\newtheorem{exam}[thm]{Example}
\newtheorem*{mainthm}{Main Theorem}
\newtheorem{thmx}{Theorem}
\renewcommand{\thethmx}{\Alph{thmx}} 
\newtheorem{corx}[thmx]{Corollary}

\theoremstyle{definition}
\newtheorem*{defi}{Definition}
\newtheorem*{rmk}{Remark}
\newtheorem*{ques}{Question}
\newtheorem*{conj}{Conjecture}
\newtheorem*{nota}{Notations}
	
\newcommand{\ind}{\mathop{\mathrm{index}}}
\newcommand{\R}{\mathbb{R}}
\newcommand{\Z}{\mathbb{Z}}
\newcommand{\N}{\mathbb{N}}
\newcommand{\J}{\mathfrak{J}}
\newcommand{\C}{\mathbb{C}}
\newcommand{\D}{\mathbb{D}}
\newcommand{\EC}{\mathbb{\widehat{C}}}

\newcommand{\Area}{\textup{Area}}
\newcommand{\MA}{\mathcal{A}}
\newcommand{\MC}{\mathcal{C}}
\newcommand{\MH}{\mathcal{H}}
\newcommand{\MM}{\mathcal{M}}
\newcommand{\MO}{\mathcal{O}}
\newcommand{\MP}{\mathcal{P}}
\newcommand{\MS}{\mathcal{S}}
\newcommand{\MU}{\mathcal{U}}
\newcommand{\Crit}{\textup{Crit}}
\newcommand{\ii}{\textup{i}}

\renewcommand{\labelenumi}{\textup{(\alph{enumi})}}

\renewcommand\theequation{\thesection.\arabic{equation}} 
\makeatletter\@addtoreset{equation}{section}\makeatother

\author{Xiaole He, Yingqing Xiao and Fei Yang}

\address{School of Mathematics, Hunan University, Changsha, 410082, P.R. China}
\email{hxlvcl@foxmail.com}

\address{School of Mathematics, Hunan University, Changsha, 410082, P.R. China}
\email{ouxyq@hnu.edu.cn}

\address{School of Mathematics, Nanjing University, Nanjing, 210093, P.R. China}
\email{yangfei@nju.edu.cn}
	
\title[Cantor bubble Julia sets]{Rational maps with Cantor bubble Julia sets}

\begin{abstract}
It has been shown that Cantor bubble Julia sets can appear in the dynamics of polynomials and their singular perturbations. In this paper, we present a criterion that guarantees the existence of Cantor bubble Julia sets for certain rational maps with attracting or parabolic fixed points. Moreover, we construct other Cantor bubble Julia sets, including those with high-periodic attracting cycles and those with Hausdorff dimension two. Finally, we give a sufficient condition for Cantor bubble Julia sets to be quasisymmetrically equivalent to Cantor round bubbles.
\end{abstract}

\date{\today}


\maketitle

\section{Introduction}

\subsection{Backgrounds}

The study of topological and geometric properties of Julia sets of rational maps is one of the important topics in complex dynamics.
It is known that for the unicritical polynomial $z\mapsto z^n$ with $n\geq 2$, its dynamics is well understood and the Julia set is the simple unit circle. However, when a small perturbation is made, then the dynamics may become very complicated.

The first example of singular perturbation was made by McMullen in 1988 \cite{McM88}. He considered a family of rational maps (now termed as \textit{McMullen family})
\begin{equation}
f_\lambda(z)=z^n+\lambda/z^m,
\end{equation}
where $n\geq 2$, $m\geq 1$ and $\lambda\in\C\setminus\{0\}$, and proved that if $1/n+1/m<1$ and $\lambda\neq 0$ is small enough, then the Julia set $J(f_\lambda)$ is a \emph{Cantor set of circles}. In particular, $J(f_\lambda)$ consists of uncountably many nested Jordan curves (see also \cite{DLU05}). One may find more such Julia sets in \cite{QYY15}.

For any $n\geq 2$ and $m\geq 1$, if $\lambda$ is large, then $J(f_\lambda)$ is a \textit{Cantor set}. In particular, each Julia component of $f_\lambda$ is a point. There are infinitely many hyperbolic components, called \textit{Sierpi\'{n}ski holes} in the parameter plane of $f_\lambda$, which correspond to \textit{Sierpi\'{n}ski carpet} Julia sets (see \cite{DLU05}, \cite{Ste06b}, \cite{Roe06a} and \cite{DP09}). In fact, if $1/n+1/m\geq 1$ (this happens only when $n=m=2$, or $n\geq 2$ and $m=1$), there exists arbitrarily small $\lambda\neq 0$ such that $J(f_\lambda)$ is a Sierpi\'{n}ski carpet (see \cite{DG08} and \cite{MD08}).
As a different topological type, the \textit{generalized Sierpi\'{n}ski gaskets} can also appear as the Julia sets of $f_\lambda$ (see \cite{DRS07} and \cite{HXY25}).

\medskip
In \cite{DM08}, Devaney and Marotta studied another singular perturbation of $z^n$ and considered the family
\begin{equation}
f_{\lambda, a}(z)=z^n+\frac{\lambda}{(z-a)^m},
\end{equation}
where $n\geq 2$, $m\geq 1$ and $\lambda, a\in\C\setminus\{0\}$.
They proved that if $|a|\neq 0,1$ and $\lambda\neq 0$ is small enough, then the Julia set of $f_{\lambda, a}$ consists of countably many disjoint Jordan curves and uncountably many point components, and the Fatou set $F(f_{\lambda, a})$ of $f_{\lambda, a}$ has exactly one multiply connected component which is completely invariant.
This is a different type of Julia sets which cannot appear in the dynamics of McMullen family. See Figure \ref{Fig:Julia-Devaney-Marotta} for an example.

\begin{figure}[!htpb]
  \setlength{\unitlength}{1mm}
  \setlength{\fboxsep}{0pt}
  \centering
  \fbox{\includegraphics[width=0.43\textwidth]{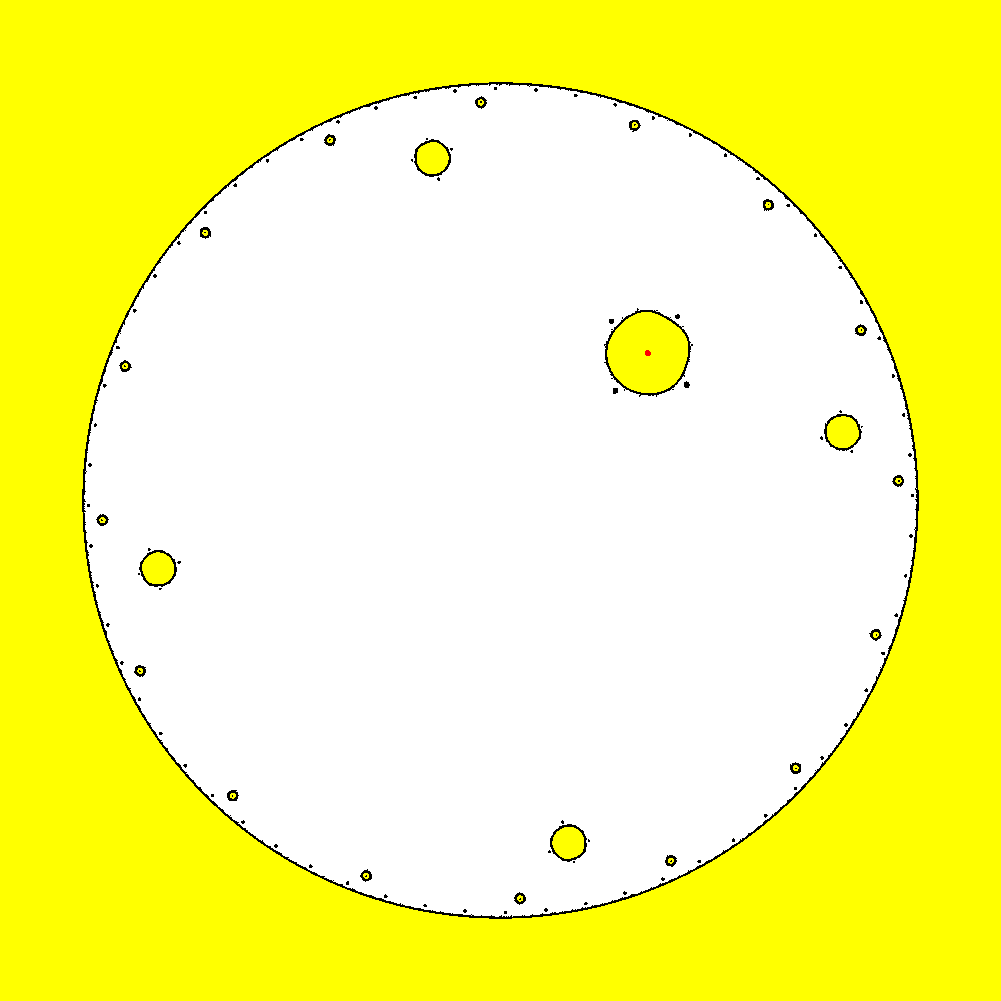}}\put(-18.2,38.7){$a$}\quad
  \fbox{\includegraphics[width=0.43\textwidth]{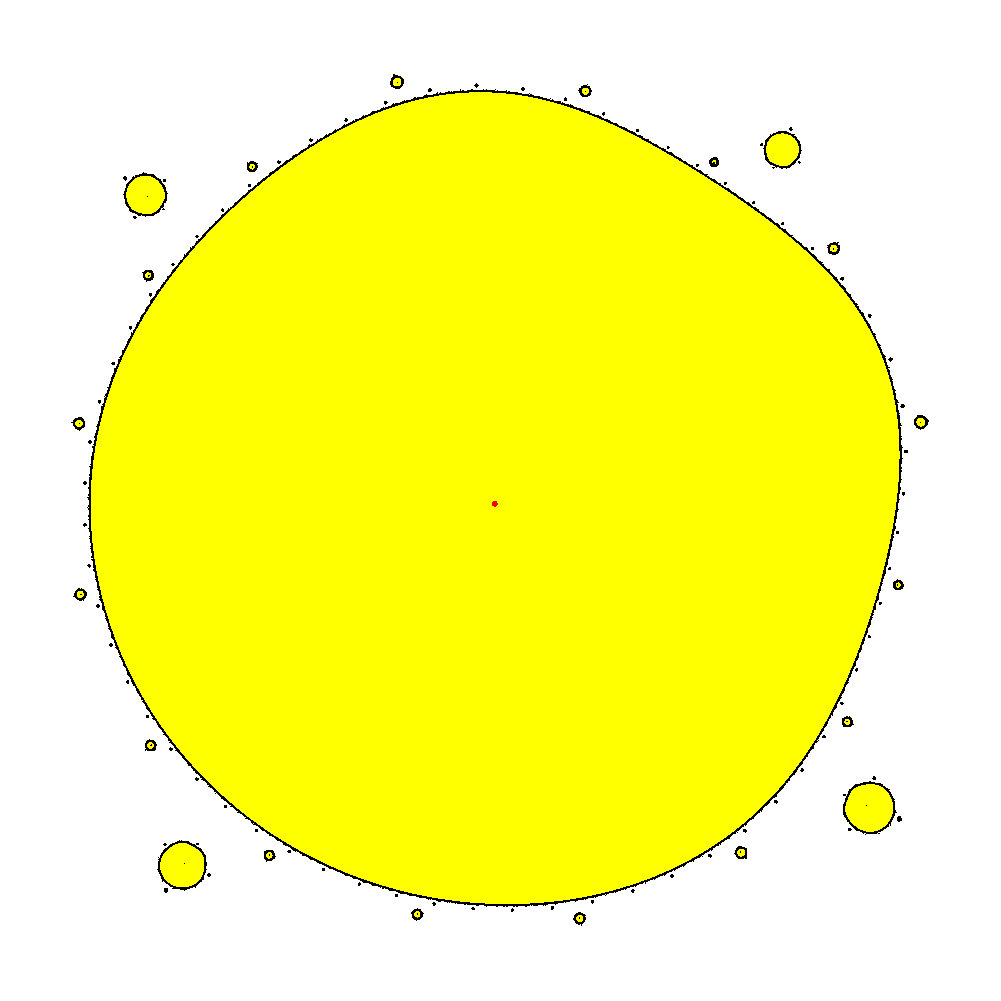}} \put(-32,27){$a$}
  \caption{The Julia set of the Devaney-Marotta family $f_{\lambda,a}(z)=z^4+\lambda/(z-a)^4$ and its zoom around the pole $a$, where $a=\frac{1}{2}e^{\pi i/4}$ and $\lambda=10^{-4}e^{4\pi i/3}$. It can be seen clearly that $F(f_{\lambda,a})$ contains exactly one multiply connected component and $J(f_{\lambda,a})$ consists of infinitely many Jordan curves and points.}
  \label{Fig:Julia-Devaney-Marotta}
\end{figure}

Later, such type of Julia sets was also found in the family $f_{\lambda,a}$ with $|a|=1$ \cite{GM10} and in many other families of rational maps. For examples, the singular perturbation of $z^n$ with two poles \cite{Mar08b} and the \textit{generalized McMullen maps}
\begin{equation}
F_{a, b}(z)=z^n+\frac{a}{z^n}+b,
\end{equation}
where $n\geq 2$ and $a, b \in \mathbb{C}\setminus\{0\}$. Specifically, it was proved in \cite{XQY14} that if both free critical values of $F_{a, b}$ lie in the same invariant bounded Fatou component, then the Julia set of $F_{a,b}$ share an identical topological structure as above.

\medskip
For the convenience of describing such type of Julia sets, we introduce the following concept, which, to our knowledge, first appeared in \cite{XQY14}.

\begin{defi}[Cantor bubbles]
A perfect subset $X$ of the Riemann sphere $\EC$ is called a \textit{Cantor set with bubbles} (\textit{Cantor bubbles} in short) if it is the disjoint union of countably many Jordan curves and uncountably many points, and its complement $\EC\setminus X$ contains exactly one multiply connected component.
\end{defi}

If the Julia set $J(f)$ of a rational map $f$ is homeomorphic to a Cantor set with bubbles, then we call $J(f)$ a \textit{Cantor bubble Julia set}.
Note that Cantor bubble Julia sets can be seen as a mixture of Cantor circle Julia sets and Cantor Julia sets in some sense.
We would like to mention that the Cantor bubbles were also called \textit{semi-Cantor sets} recently (see \cite{Etk22}, \cite{HE25}).

\subsection{Main results}

In view of that the proof methods for the existence of Cantor bubble Julia sets above depend on special families, it is meaningful to give a general criterion to judge whether a
given rational map has a Cantor bubble Julia set.

\begin{thm}\label{thm:criterion}
Let $f$ be a rational map of degree $d\geq 2$. Suppose
\begin{enumerate}
\item $f$ has two invariant Fatou components $U$ and $V$;
\item $U$ is completely invariant while $V$ is not; and
\item $U\cup V$ contains all critical values of $f$.
\end{enumerate}
Then the Julia set $J(f)$ of $f$ is a Cantor set with bubbles.
\end{thm}

By analyzing the critical orbits and dynamics in the Fatou components of the rational families mentioned before, Theorem \ref{thm:criterion} covers the corresponding results in \cite{DM08}, \cite{Mar08b}, \cite{XQY14} and \cite{HE25}.
As an immediate application of Theorem \ref{thm:criterion}, we conclude that if a cubic polynomial $f$ has an attracting fixed point and a critical point in $\C$ escaping to $\infty$, then $J(f)$ is a Cantor set with bubbles (this fact is known to some experts). See Figure \ref{Fig:Julia-Cantor bubble}.

\begin{figure}[!htpb]
  \setlength{\unitlength}{1mm}
  \centering
  \includegraphics[width=0.85\textwidth]{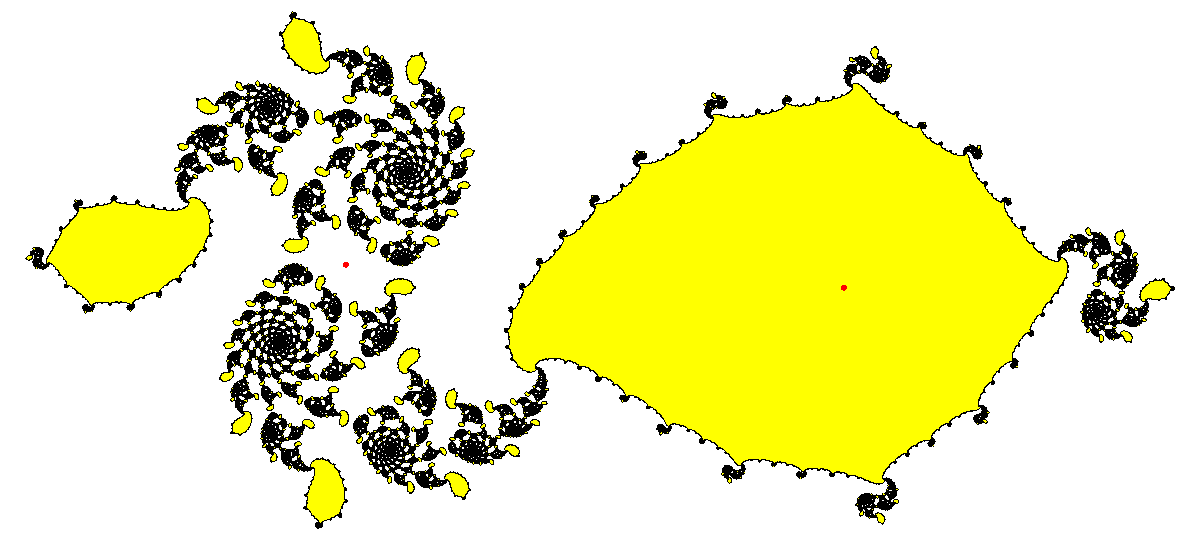}
  \put(-36.7,21.5){$0$}
  \put(-87.5,25.2){$-a$}
  \caption{The Julia set of $f(z)=\frac{3}{2}az^2+z^3$ with $a=0.06+1.31i$ is a Cantor set with bubbles, where $f$ has a superattracting fixed point at $0$ and an escaping critical point at $-a$ (the picture has been rotated).}
  \label{Fig:Julia-Cantor bubble}
\end{figure}

In view of the periodic Fatou components of the known rational maps having Cantor bubble Julia sets all have period one, a meaningful question is to explore whether such Julia sets can appear in the rational maps with high periodic Fatou components. The following result gives an affirmative answer to this question.

\begin{thm}\label{thm:period}
For any integer $p\geq 1$, there exist cubic polynomials having a superattracting cycle of period $p$ whose Julia sets are Cantor bubbles.
\end{thm}

The space $S_p:=\{$cubic polynomials having a superattracting $p$-cycle in $\C\}$ has been studied extensively (see \cite{Mil09}, \cite{BKM10}, \cite{Wan21} and the references therein). However, the existence of Cantor bubble Julia sets in all $S_p$'s was not guaranteed.
In fact, taking $p=2$ as an example, we need to be careful of the appearance of the ``basilica structure" in the Julia sets, i.e., it may happen that a Julia component is homeomorphic to the Julia set of $z\mapsto z^2-1$, which is not a Jordan curve (see Figure \ref{Fig:Julia-Cantor bubble-p}). To this end, we first construct topological cubic polynomials with the desired combinations and then use Thurston's theory to prove that these topological polynomials can be realized by genuine polynomials in Theorem \ref{thm:period}.

\medskip

According to Przytycki and Zdunik, if a polynomial has a disconnected but not totally disconnected Julia set, then the Hausdorff dimension of this Julia set is strictly larger than $1$ \cite{PZ21}. Therefore, all polynomial Cantor bubble Julia sets have Hausdorff dimension strictly larger than $1$. On the other hand, according to Shishikura \cite[Theorem 2]{Shi98}, the Hausdorff dimension of Cantor bubble Julia sets can be arbitrarily close to $2$. In this paper, we prove the following result.

\begin{thm}\label{thm:dim-2}
There exist Cantor bubble Julia sets of polynomials having Hausdorff dimension two.
\end{thm}

A candidate approach to prove Theorem \ref{thm:dim-2} is to follow the idea in \cite{Yan21}, where the existence of Cantor Julia sets with Hausdorff dimension two was proved.
However, the arguments there depend heavily on the parameter structure of cubic polynomials obtained in \cite{BH92}.
As an alternative, in this paper we use quasiconformal surgery to ``put" a special Cantor Julia set in the Julia sets of degree five polynomials such that they are Cantor bubble Julia sets having full Hausdorff dimension.

\medskip
The quasisymmetric uniformization of fractal sets is one of the central problems in the quasiconformal geometry (see \cite{Bon06}, \cite{Kle06}, \cite{Nta26}).
For Julia sets of rational maps, the quasisymmetric uniformizations of Cantor Julia sets, Sierpi\'{n}ski carpet Julia sets, Sierpi\'{n}ski gasket Julia sets and Cantor circle Julia sets have been studied in \cite[Proposition 15.11]{DS97}, \cite{BLM16}, \cite{QYZ19}, \cite{QY21}, \cite{LN24}, \cite{LZ25} respectively.
In this paper, we study the quasisymmetric uniformization of Cantor bubble Julia sets. Inspired by Sierpi\'{n}ski carpets, we introduce the following concept.

\begin{defi}[Cantor round bubbles]
A subset $X$ of $\EC$ is called a \textit{Cantor set with round bubbles} (\textit{Cantor round bubbles} in short) if $X$ is homeomorphic to a Cantor set with bubbles and each Jordan curve component of $X$ is a spherical circle.
\end{defi}

We provide a sufficient condition for Cantor bubble Julia sets to be quasisymmetrically equivalent to Cantor round bubbles.

\begin{thm}\label{thm:roundbubbles}
Let $f$ be a rational map having a Cantor bubble Julia set $J(f)$. If the critical orbit of $f$ does not accumulate on each Jordan curve component of $J(f)$, then $J(f)$ is quasisymmetrically equivalent to a Cantor set with round bubbles.
\end{thm}

For the definition of quasisymmetric equivalence, see \S\ref{sec:quasis}.
In fact, an orientation preserving homeomorphism from $\EC$ to itself is quasisymmetric if and only if it is quasiconformal \cite[Theorem 11.14]{Hei01}.
The proof of Theorem \ref{thm:roundbubbles} relies on Bonk's criterion \cite{Bon11} (see Theorem \ref{thm:Bonk}). We will show that the collection of all Jordan curve components of $J(f)$ are uniform quasicircles and uniformly relatively separated.

We would like to mention that in the cases of Sierpi\'{n}ski carpet and Cantor circle Julia sets, any parabolic point is not allowed if they are quasisymmetrically equivalent to the round objects. However, for Cantor bubble Julia sets, we will see that some Cantor bubble Julia set with parabolic points can still be quasisymmetrically equivalent to Cantor round bubbles.

\medskip
This paper is organized as following:
In \S\ref{sec:criterion}, we give a proof of Theorem \ref{thm:criterion} based on the classification of periodic Fatou components and the straightening theory on periodic Julia components due to McMullen.
In \S\ref{sec:period}, we use Thurston's theory to construct cubic polynomials with Cantor bubble Julia sets having a superattracting cycle of any given period (Theorem \ref{thm:period}).
In \S\ref{sec:dim}, we use quasiconformal surgery to glue a Cantor Julia set and a superattracting basin to obtain Cantor bubble Julia sets with full Hausdorff dimension (Theorem \ref{thm:dim-2}).
In \S\ref{sec:quasis}, we use distortion estimates to obtain the uniform quasicircles and uniformly relatively separated properties, and Theorem \ref{thm:roundbubbles} is proved by applying Bonk's criterion.

\medskip
\noindent\textbf{Acknowledgements.}
This work was supported by NSFC (Grant Nos.\,12471073, 12571093) and NSF of Hunan Province (Grant No.\,2025JJ50049).

\section{A criterion to generate Cantor bubbles}\label{sec:criterion}

In this section we give a proof of Theorem \ref{thm:criterion} and apply the criterion to study a parabolic rational map with a Cantor bubble Julia set.

\subsection{Proof of Theorem \ref{thm:criterion}}

To study the dynamics of disconnected Julia sets, the following result is very useful.

\begin{thm}[{\cite[Theorem 3.4]{McM88}}]\label{thm:McM1988}
If a rational map $f$ has a non-singleton Julia component $J_0$ of period $p\geq 1$, then there exist a rational map $g$ and a quasiconformal mapping $\phi: \EC\rightarrow \EC$ such that $\phi(J_0)=J(g)$ and $\phi\circ f^{\circ p}=g\circ \phi$ holds on $J_0$.
\end{thm}

This result implies that one may extract rational maps with connected Julia sets from the ones with disconnected Julia sets.

\begin{proof}[Proof of Theorem \ref{thm:criterion}]
\textit{Step 0}. Let $U$ and $V$ be two invariant Fatou components of $f$ such that $U$ is completely invariant while $V$ is not.
Since $U$ is completely invariant, we have $\partial U =J(f)$ and each component of $\EC\setminus\overline{U}$ is a simply connected Fatou component.

\medskip
\textit{Step 1}. \textit{Two straightenings}. Let $J_V$ be the connected component of $J(f)$ containing $\partial V$. Since $V$ is invariant, we have $f(J_V) = J_V$.
By Theorem \ref{thm:McM1988}, there exist a rational map $g$ with $\deg(g)\leq\deg(f)$  and a quasiconformal mapping $\phi: \widehat{\mathbb{C}} \to \widehat{\mathbb{C}}$ such that
\begin{equation}
\phi(J_V)=J(g) \text{\quad and\quad}\phi\circ f=g\circ \phi \text{ on } J_V.
\end{equation}

Let $K_V$ be the connected component of $\EC\setminus U$ containing $V$. Then $f(K_V) = K_V$. Let $V_0$ be any component of $K_V\setminus J_V$. Then $V_0$ is a simply connected Fatou component of $f$. Since the quasiconformal mapping $\phi$ is orientation-preserving, it follows that $\phi(V_0)$ is a Fatou component of $g$ and $\phi(V_0)\neq\EC\setminus \phi(K_V)$. Thus $\EC\setminus \phi(K_V)$ is a completely invariant Fatou component of $g$. Note that $\deg(f:V_0\to f(V_0))=\deg(g:\phi(V_0)\to \phi(f(V_0)))$. Since $U\cup V$ contains all critical values of $f$, it follows that $\phi(V)\cup(\EC\setminus \phi(K_V))$ contains all critical values of $g$. Hence $g$ is a geometrically finite rational map having a connected Julia set.

By a quasiconformal surgery procedure (see \cite[p.\,106, Theorem 5.1]{CG93} and \cite[Proposition 6.8]{McM88}), there exist a rational map $h$ with $\deg(h)=\deg(g)$ and a quasiconformal mapping $\psi: \widehat{\mathbb{C}} \to \widehat{\mathbb{C}}$ such that
\begin{equation}
\psi(J(g))=J(h) \text{\quad and\quad}\psi\circ g=h\circ \psi \text{ on } J(g),
\end{equation}
and moreover, all critical values of $h$ are contained in the invariant Fatou components
\begin{equation}
V':=\psi(\phi(V)) \text{\quad and\quad} V'':=\psi(\EC\setminus \phi(K_V))
\end{equation}
of $h$, and each of them contains exactly one critical value of $h$. Therefore, $h$ is a rational map having exactly two critical values. By the Riemann-Hurwitz formula, $h$ has exactly two critical points (without counting multiplicity) and they have the same local degree. This implies that $V'$ and $V''$ are completely invariant under $h$ and hence $h$ has exactly two Fatou components, which are attracting or parabolic basins. By \cite{Ste97}, $J(h)$ is a Jordan curve. Therefore, $J(g)$ and $J_V$ are Jordan curves.

\medskip
\textit{Step 2}. \textit{Countably many Jordan curves}.
Note that $\partial V = J_V$ is a Jordan curve component. Since $V$ is not completely invariant, we conclude that $J(f)$ is disconnected and $f^{-1}(V)\setminus V\neq\emptyset$.
By the assumption that $U\cup V$ contains all critical values of $f$, it follows that $\partial V$ is disjoint with the critical orbit of $f$. Hence $\bigcup_{n\geq 0}f^{-n}(V)$ consists of countably many Jordan domains with pairwise disjoint boundaries.

\medskip
\textit{Step 3}. \textit{Rest components are points}.
Let $J_0$ be a Julia component of $f$ which is not a preimage of $\partial V$. In the following we prove that $J_0$ is a point.
Let $K_0$ be the component of $\EC\setminus U$ containing $J_0$. Then
\begin{equation}\label{equ:f-orbit-empty}
f^{\circ n}(K_0)\cap(U\cup\overline{V})=\emptyset \text{\quad for all } n\geq 0.
\end{equation}
Note that $f$ is geometrically finite. If $J_0$ is not a point, according to \cite[Theorem 1.2]{PT00}, either $J_0$ is a Jordan curve, or $J_0$ is eventually periodic.
If $J_0$ is a Jordan curve, then $K_0$ is a closed Jordan disk. By Sullivan's no wandering domain theorem, there exists $n\geq 0$ such that $f^{\circ n}(K_0)$ is periodic. Therefore, if $J_0$ is not a point, then there exists $n_0\geq 0$ such that $f^{\circ n_0}(J_0)$ and $f^{\circ n_0}(K_0)$ are periodic in both cases.
By Theorem \ref{thm:McM1988} and the same arguments above, $f^{\circ n_0}(K_0)$ intersects with the critical orbit of $f$.
By \eqref{equ:f-orbit-empty}, this is a contradiction since the invariant set $U\cup V$ contains all critical values of $f$. Hence $J_0$ is a point.

\medskip
\textit{Step 4}. \textit{The conclusion}.
Note that any disconnected Julia set consists of uncountable many Julia components (see \cite[Corollary 4.15]{Mil06}). It follows that $J(f)$ is the disjoint union of countably many Jordan curves and uncountably many points, and its complement $\EC\setminus J(f)$ contains exactly one multiply connected component $U$. Thus $J(f)$ is a Cantor set with bubbles.
\end{proof}

\begin{rmk}
In Step 3 of the proof of Theorem \ref{thm:criterion}, we can also apply the result of the generalized Branner-Hubbard conjecture \cite{Zha10} to conclude that the Julia components which are not iterated to $\partial V$ are points.

If the third condition in Theorem \ref{thm:criterion} is not satisfied, then $J(f)$ may be connected. Indeed, for example, if a cubic polynomial $f$ has a fixed superattracting basin $V$ and the other critical point of $f$ is not in $V$ but its forward orbit intersects with $V$, then $J(f)$ is connected.
\end{rmk}

We list some basic properties of the rational maps with Cantor bubble Julia sets.

\begin{lem}\label{lem:attr-or-para}
Let $f $ be a rational map with a Cantor bubble Julia set. Then
\begin{enumerate}
\item $\deg(f)\geq 3$ and the critical points in $J(f)$ are contained in point components of $J(f)$; and
\item $f$ has exactly one completely invariant Fatou component and all periodic Fatou components of $f$ are attracting or parabolic.
\end{enumerate}
\end{lem}

\begin{proof}
(a) Note that the Julia set of any quadratic rational map is either connected or totally disconnected (i.e., a Cantor set, see \cite{Mil93}, \cite{Yin92}). Hence $\deg(f)\geq 3$.
If a Jordan curve Julia component contains a critical point, then there exists a ``Figure 8" structure near this critical point in $J(f)$, which is a contradiction.

\medskip
(b) Let $U$ be the only multiply connected Fatou component of $f$. It is infinitely connected and hence is completely invariant since other Fatou components are simply connected. In particular, $U$ is an attracting or a parabolic basin. Obviously, $f$ has no Herman rings since $f$ has no doubly connected Fatou components.

By Sullivan's no wandering domain theorem, each component of $\EC\setminus \overline{U}$ is a Jordan Fatou component which is eventually periodic. Let $W$ be a periodic component of $\EC\setminus\overline{U}$ with period $p\geq 1$. By Theorem \ref{thm:McM1988}, there exist a rational map $g$ and a quasiconformal mapping $\phi: \widehat{\mathbb{C}} \to \widehat{\mathbb{C}}$ such that
\begin{equation}
\phi(\partial W)=J(g) \text{\quad and\quad}\phi\circ f^{\circ p}=g\circ \phi \text{ on } \partial W.
\end{equation}
Thus $J(g)$ is a Jordan curve and $g$ has no Siegel disks. Therefore, $f$ has no rotation domains and all periodic Fatou components of $f$ are attracting or parabolic.
\end{proof}

\subsection{Specific examples I}

By Lemma \ref{lem:attr-or-para}, the completely invariant Fatou component of a rational map with a Cantor bubble Julia set is either attracting or parabolic.
In this subsection, we apply Theorem \ref{thm:criterion} to study a cubic rational map such that it has a completely invariant parabolic basin and a Cantor bubble Julia set.

\begin{exam}[See Figure \ref{Fig:Julia-Cantor-bubble-attr-para-2}]\label{exam:bubble-para}
Consider the cubic rational family
\begin{equation}
f_a(z)=b\,\frac{z^3+cz+d}{z+a},
\end{equation}
where
\begin{equation}
b=\frac{(1+a)^2}{2+3a}, \quad c=\frac{1+2a}{(1+a)^2} \text{\quad and\quad} d=\frac{a(1+2a)}{(1+a)^2}
\end{equation}
with $a\in\C\setminus\{0,-1,-\frac{2}{3}\}$. Then $f_a$ has a superattracting fixed point $\infty$ and a parabolic fixed point $1$ with multiplier $1$.
If $-1-a>0$ is small enough, then $J(f_a)$ is a Cantor set with bubbles.
\end{exam}

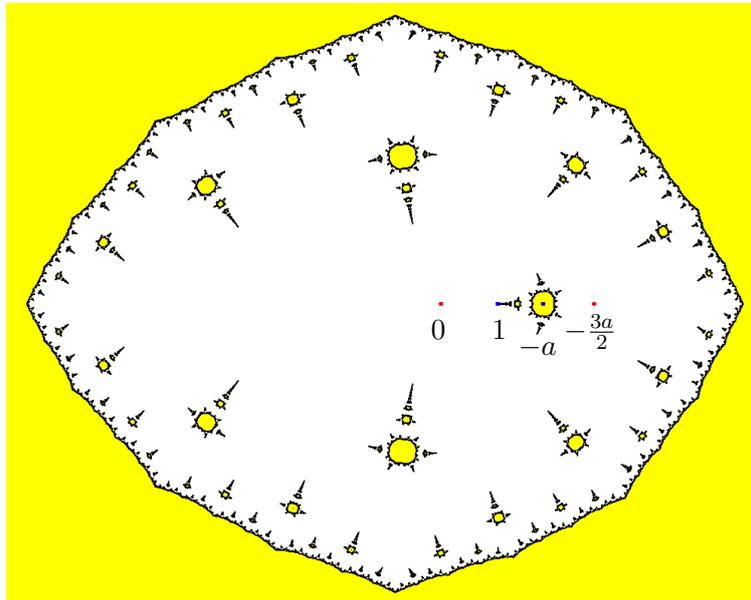
\begin{figure}[!htpb]
  \setlength{\unitlength}{1mm}
  \centering
  \includegraphics[width=0.7\textwidth]{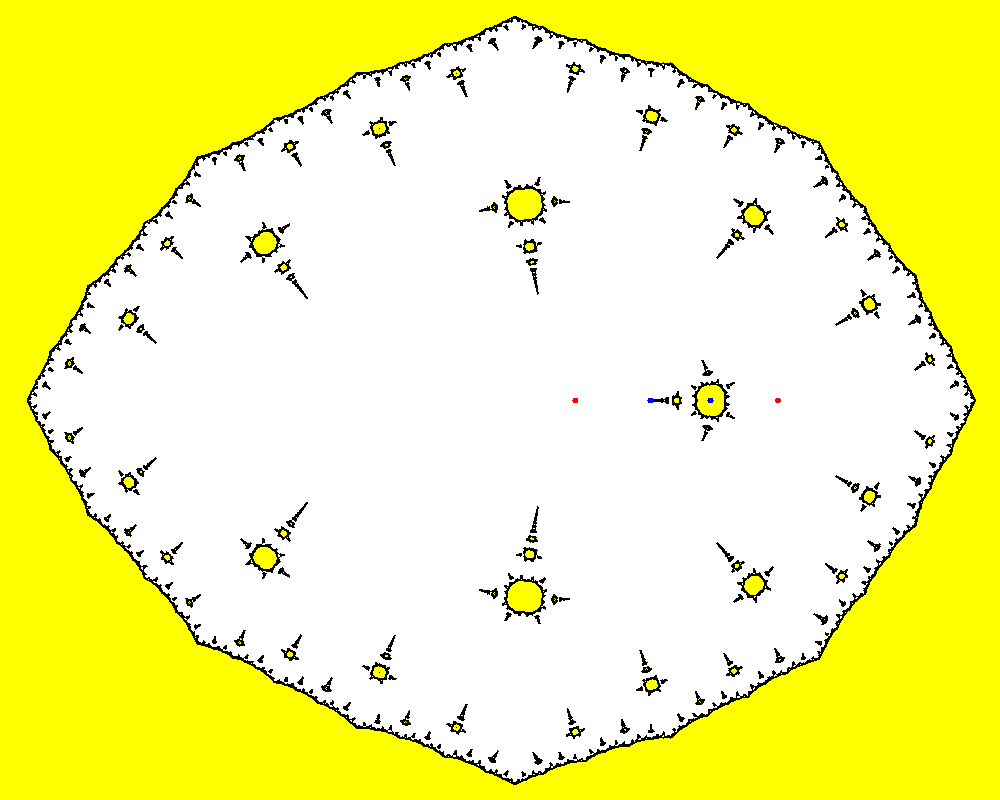}
  \put(-43.5,35){$0$}
  \put(-35.5,35){$1$}
  \put(-26,35){$-\frac{3a}{2}$}
  \put(-32,33){$-a$}
  \caption{A Cantor bubble Julia set of $f_a$ in Example \ref{exam:bubble-para}, where $f_a$ has a parabolic fixed point $1$ whose parabolic basin is infinitely connected, a double critical point $0$, a simple critical point $-\frac{3a}{2}$ and a superattracting fixed point $\infty$. Here $a=-1.8$ is chosen such that some bounded simply connected Fatou components are visible.}
  \label{Fig:Julia-Cantor-bubble-attr-para-2}
\end{figure}

\begin{proof}
By direct calculations, $f_a$ has a superattracting fixed point $\infty$, a parabolic fixed point $1$ with multiplier $1$, a double critical point $0$ and a simple critical point $-\frac{3a}{2}$.
If $a<-1$ but sufficiently close to $-1$, we will verify the following properties:
\begin{enumerate}
\item $f_a$ is strictly increasing on $[0,1]$ and $f_a([0,1])=[f_a(0),1]\subset(0,1]$;
\item $f_a(z)\neq z$ for all $z\in [0,1)$; and
\item $0<f_a(-\frac{3a}{2})<1$.
\end{enumerate}
These properties imply that the critical values $f_a(0)$ and $f_a(-\frac{3a}{2})$ are contained in the immediate parabolic basin $U$ of $1$. Since $f_a$ has local degree three at $0$, hence $U$ is completely invariant and $0,-\frac{3a}{2}\in U$.
Let $V$ be the Fatou component of $f_a$ containing $\infty$. We claim that $V$ cannot be completely invariant. Indeed, note that $f_a$ is a real rational map for real $a$. Every Fatou component of $f_a$ is symmetric about the real axis if it intersects with the real axis. If the pole $-a$ is contained in $V$, then $-\frac{3a}{2}$ and $0$ are separated by $V$, which is a contradiction. Hence $-a\not\in V$ and $V$ is not completely invariant.
By Theorem \ref{thm:criterion}, $J(f_a)$ is a Cantor set with bubbles.

\medskip
It is sufficient to verify the three properties mentioned above. A direct calculation shows that
\begin{equation}
f_a(0)=\frac{1+2a}{2+3a} \text{\quad and\quad} f_a'(z)=b\,\frac{z^2(2z+3a)}{(z+a)^2}.
\end{equation}
If $z\in(0,1)$ and $a<-1$ but sufficiently close to $-1$, then $f_a(0)\in(0,1)$ and $f_a'(z)>0$. Hence $f_a$ is strictly increasing on $[0,1]$ and $f_a([0,1])=[f_a(0),1]\subset(0,1]$.

By a direct calculation, we have
\begin{equation}
f_a(z)-z=\frac{(z-1)^2((a+1)^2 z+2a^2+a)}{(3a+2)(z+a)}
\end{equation}
and
\begin{equation}
f_a(-\tfrac{3a}{2})=1+\tfrac{1}{4}(a+1)(3a+2)(3a-1).
\end{equation}
Hence if $z\in(0,1)$ and $a<-1$ but sufficiently close to $-1$, then $f_a(z)\neq z$ and $0<f_a(-\frac{3a}{2})<1$. The proof is complete.
\end{proof}

\section{More Cantor bubble Julia sets}\label{sec:period}

Thurston's topological characterization theorem is an effective tool to construct rational maps with prescribed dynamics.
The original theorem \cite{DH93} was proved for postcritically finite maps, and the theory was later extended to certain topological polynomials \cite{BFH92} and some postcritically infinite cases \cite{CT11}, \cite{ZJ09}.
In this section we use some of these theories to construct more Cantor bubble Julia sets.

\subsection{Topological characterization of semi-rational maps}

Let $f:\EC\to\EC$ be a branched covering of $\deg(f)\geq 2$ and let $\text{Crit}(f)$ be the set of all critical points of $f$. The \textit{postcritical set} of $f$ is
\begin{equation}
P(f):=\overline{\bigcup_{n\geq 1 } f^{\circ n}\big(\Crit(f)\big)}.
\end{equation}
The map $f$ is \textit{postcritically finite} if $P(f)$ is a finite set (i.e., every critical point of $f$ has a finite orbit).
The map $f$ is a \textit{semi-rational map} (see \cite{CT11}, \cite{ZJ09}) if the accumulation set $P'(f)$ of $P(f)$ is finite (or empty), and in the case $P'(f)\neq\emptyset$, the map $f$ is
holomorphic in a neighborhood of  $P'(f)$ and every periodic point in $P'(f)$ is either attracting or superattracting.

\begin{defi}[c-equivalent]
Two semi-rational maps $f$ and $g$ are called \textit{c-equivalent}, if there exist a pair $(\phi, \psi)$ of homeomorphisms of $\EC$ to itself and a neighborhood $U_0$ of $P'(f)$ such that:
\begin{enumerate}
\item $\phi \circ f=g \circ \psi$;
\item $\phi$ is holomorphic in $U_0$;
\item $\phi$ and $\psi$ are equal on $P(f)$, thus on a neighborhood of $P'(f)$; and
\item $\phi$ and $\psi$ are isotopic to each other rel $P(f) \cup \overline{U}_0$.
\end{enumerate}
\end{defi}

A \textit{multicurve} $\Gamma$ on $(\EC,P(f))$ is a finite set of disjoint, non-homotopic simple closed curves on $\EC\setminus P(f)$ such that both components of the complement of every curve $\gamma\in\Gamma$ contains at least two points of $P(f)$. Such curves are called \textit{non-peripheral}. A multicurve $\Gamma$ is called \textit{$f$-stable} if for every $\gamma\in\Gamma$, each component of $f^{-1}(\gamma)$ is either homotopic rel $P(f)$ to an element of $\Gamma$ or peripheral.
An $f$-stable multicurve $\Gamma$ is called a \textit{Thurston obstruction} of $f$ if the leading eigenvalue $\lambda_\Gamma$ of the Thurston transformation matrix of $\Gamma$ satisfies $\lambda_\Gamma\geq 1$, which is in fact closely related to Gr\"{o}tzsch's inequality on moduli of annuli. See \cite{DH93} or \cite[\S10.1]{Hub16} for precise definitions.

\begin{thm}[{\cite{CT11} and \cite{ZJ09}}]\label{thm:Cui-Tan}
Let $f$ be a semi-rational map with $P'(f) \neq \emptyset$. Then the map $f$ is c-equivalent to a rational map $R$ if and only if $f$ has no \emph{Thurston obstructions}.
In this case the rational map $R$ is unique up to M\"{o}bius conjugations.
\end{thm}

\begin{defi}[Levy cycle]
A multicurve $\Gamma=\{ \gamma_0, \dots, \gamma_{n-1},\gamma_n = \gamma_0 \}$ is a \textit{Levy cycle} if for each $i = 0, \dots, n-1$, at least one component $\gamma'$ of $f^{-1}(\gamma_{i+1})$ is homotopic rel $P(f)$ to $\gamma_i$ and $f: \gamma' \to \gamma_{i+1}$ has degree 1.

The multicurve $\Gamma$ is a \textit{degenerate Levy cycle} if it is a Levy cycle together with disjoint disks $D_i$ bounded by $\gamma_i$ for each $i$, such that one component of $f^{-1}(D_{i+1})$ is a disk $D_i'$ satisfying $P(f)\cap D_i'=P(f)\cap D_i$ and $f:D_i'\to D_{i+1}$ is a homeomorphism.
\end{defi}

Levy cycles are special Thurston obstructions \cite[Theorem 5.4]{BFH92}.
A branched covering $f:\EC\to\EC$ of $\deg(f)\geq 2$ is said to be a \textit{topological polynomial} if $f^{-1}(\infty)=\{\infty\}$.
The following result was proved by Bielefeld, Fisher and Hubbard (see also \cite[Theorem 10.3.8]{Hub16}).

\begin{thm}[{\cite[Theorem 5.5]{BFH92}}]\label{thm:BFH92}
If $f$ is a topological polynomial having a Thurston obstruction $\Gamma$, then $f$ has a degenerate Levy cycle $\Gamma'\subset\Gamma$.
\end{thm}

In \cite{BFH92} and \cite[\S 10]{Hub16}, Thurston obstructions are defined only for postcritically finite maps. However, the proof for Theorem \ref{thm:BFH92} also applies to postcritically infinite case, in particular, to semi-rational maps.
If a topological polynomial $f$ is also semi-rational, then we call $f$ a \textit{semi-polynomial}.

\begin{cor}\label{cor:Bers-Levy}
Let $f$ be a semi-polynomial such that every critical point of $f$ is either attracted by $\infty$, or lands in a periodic cycle that contains a critical point. Then $f$ is c-equivalent to a polynomial.
\end{cor}

\begin{proof}
The proof is inspired by the Levy-Berstein Theorem for postcritically finite topological polynomials (see \cite[Corollary 5.13]{BFH92} or \cite[Theorem 10.3.9]{Hub16}).
Assume that $f$ has a Thurston obstruction $\Gamma$. By Theorem \ref{thm:BFH92}, $\Gamma$ contains a degenerate Levy cycle $\Gamma'=\{ \gamma_0, \dots, \gamma_{n-1},\gamma_n = \gamma_0 \}$ bounding disjoint topological disks $\{D_0, \dots, D_{n-1},D_n = D_0 \}$ such that one component of $f^{-1}(D_{i+1})$ is a disk $D_i'$ satisfying that $f:D_i'\to D_{i+1}$ is a homeomorphism and moreover,
\begin{equation}
f(P(f)\cap D_i')=f(P(f)\cap D_i)=P(f)\cap D_{i+1}, \text{\quad for all } i=0,\dots, n-1.
\end{equation}
By the assumption of the corollary, at least one of the $D_i$ contains a critical point, and then the component of $f^{-1}(\gamma_{i+1})$ homotopic to $\gamma_i$ cannot map to $\gamma_{i+1}$ of degree $1$, which is a contradiction. Hence $f$ has no Thurston obstruction.
By Theorem \ref{thm:Cui-Tan}, $f$ is c-equivalent to a polynomial since $f$ is topological polynomial.
\end{proof}

\subsection{Proof of Theorem \ref{thm:period}}

The criterion in Theorem \ref{thm:criterion} is applied to rational maps having periodic Fatou components with period $1$.
We shall use the topological characterization criterion obtained in previous subsection to generate more Cantor bubble Julia sets.

\begin{proof}[Proof of Theorem \ref{thm:period}]
\textit{Step 0}. \textit{The $p=1$ case}. Let $f_0(z)=-3z^2+z^3$ be a cubic polynomial having a superattracting fixed point $0$. Since $f_0$ has a critical point $2$ whose forward orbit $2\mapsto -4\mapsto -112\mapsto\cdots$ is attracted by $\infty$, by Theorem \ref{thm:criterion}, $J(f_0)$ is a Cantor set with bubbles (compare Figure \ref{Fig:Julia-Cantor bubble}).

\medskip
\textit{Step 1}. \textit{Basic setting based on $p=1$}. By B\"{o}ttcher's theorem, there exists a unique conformal mapping $\varphi_0$ defined in a neighborhood of $\infty$ which is tangent to the identity at $\infty$ such that $\varphi_0$ conjugates $f_0$ to $\zeta\mapsto\zeta^3$. The \textit{\textit{potential function}}
\begin{equation}
h_0(z)=\lim_{k\to\infty}3^{-k}\log_+\big|f_0^{\circ k}(z)\big|
\end{equation}
is defined in $\EC$ and coincides with $\log|\varphi_0(z)|$ in a neighborhood of $\infty$, where $\log_+ x:=\max\{\log x,0\}$ for $x>0$. Moreover, $h_0(z)>0$ if and only if $z$ is in the superattracting basin of $\infty$. It is well known that the map $\varphi_0$ can be extended to a conformal isomorphism $\varphi_0:U_\infty\to \EC\setminus\overline{\D}_r$ for $r=e^{ h_0(2)}>1$, where
\begin{equation}
U_\infty:=\{z\in\EC:h_0(z)>h_0(2)\} \text{\quad and\quad}\D_r:=\{z\in\C:|z|<r\}.
\end{equation}
Moreover, $U_\infty$ is simply connected and $\partial U_\infty$ has figure-eight shape such that $\EC\setminus \overline{U}_\infty$ consists of two Jordan disks. We label them by $U_0$ and $U_1$ such that $0\in U_0$. Then $f_0: U_0\to V_0$ and $f_0: U_1\to V_0$ are proper maps of degrees 2 and 1 respectively, where $V_0:=\{z\in\C:h_0(z)<h_0(-4)\}$. See Figure \ref{Fig-bubble-construction}.

\begin{figure}[!htpb]
  \setlength{\unitlength}{1mm}
  \centering
  \includegraphics[width=0.75\textwidth]{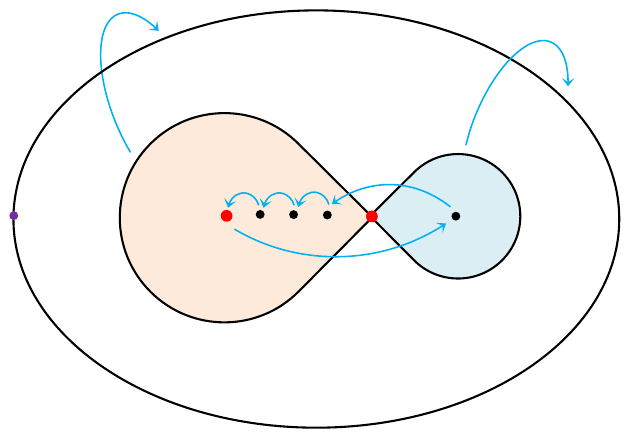}
  \put(-75,26){$U_0$}
  \put(-92,52){$f_0$}
  \put(-30,57){$f_0$}
  \put(-102.5,36.5){$-4$}
  \put(-42,45){$f$}
  \put(-55,43.2){$f$}
  \put(-60.5,43.2){$f$}
  \put(-66.3,43.2){$f$}
  \put(-25,36){$U_1$}
   \put(-30.5,33){$z_1$}
  \put(-50,26.5){$f$}
  \put(-71,34){$0$}
  \put(-52,34){\small{$z_2$}}
  \put(-57,34){\small{$z_3$}}
  \put(-64.5,34.5){\small{$z_{p-1}$}}
  \put(-44.7,33){$2$}
  \put(-25,19){$V_0$}
  \caption{A sketch of the construction of the topological polynomial $f$ such that $f|_{\overline{U}_\infty}=f_0$ and $f(z_i)=z_{i+1}$ for all $0\leq i\leq p-1$, where $z_1\in U_1$ and $z_0=z_p=0,z_2,\cdots, z_{p-1}\in U_0$.}
  \label{Fig-bubble-construction}
\end{figure}
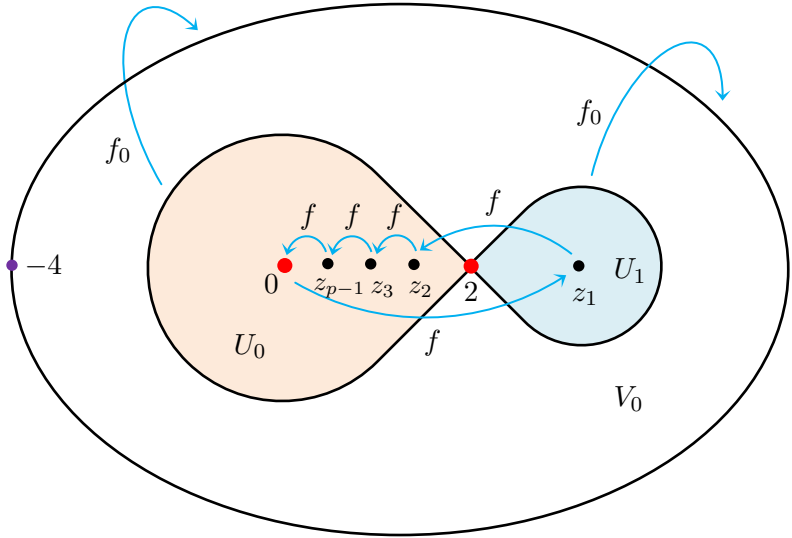

\textit{Step 2}. \textit{Construction when $p\geq 2$}. Let $p\geq 2$ be an integer. We choose $z_1\in U_1$ and $p-1$ different points $z_0=z_p=0,z_2,\cdots, z_{p-1}$ in $U_0$.
Define a branched covering $f:\EC\to\EC$ of degree 3 such that
\begin{itemize}
\item the restriction of $f$ on $\overline{U}_\infty$ is $f_0$;
\item $0$ is a critical point, $f: U_0\to V_0$ is a proper map of degree 2 and $f: U_1\to V_0$ is a homeomorphism; and
\item $f(z_i)=z_{i+1}$ for all $0\leq i\leq p-1$.
\end{itemize}
Then $f$ is a semi-polynomial such that the critical point $2$ is attracted by $\infty$ and the critical point $0$ lies in a periodic cycle. By Corollary \ref{cor:Bers-Levy}, $f$ is c-equivalent to a cubic polynomial $g$. In the following we prove that $J(g)$ is a Cantor set with bubbles.

\medskip
\textit{Step 3}. \textit{The structure of $g$}.
By the definition of $c$-equivalent, there is a pair $(\phi_0, \phi_1)$ of homeomorphisms of $\EC$ and a neighborhood $W$ of $\infty$ such that:
\begin{itemize}
\item $\phi_0 \circ f=g \circ \phi_1$;
\item $\phi_0$ is holomorphic in $W$;
\item $\phi_0$ and $\phi_1$ are equal on $P(f)\cup W$; and
\item $\phi_0$ and $\phi_1$ are isotopic to each other rel $P(f) \cup \overline{W}$.
\end{itemize}
By lifting the isotopy between $\phi_0$ and $\phi_1$, we obtain a sequence of homeomorphisms $(\phi_n:\EC\to\EC)_{n\geq 0}$ which is isotopic to $\phi_0$ rel $P(f) \cup \overline{W}$ for all $n\geq 0$ such that
\begin{itemize}
\item $\phi_n \circ f=g \circ \phi_{n+1}$; and
\item $\phi_{n+1}=\phi_n$ on $f^{-n}(P(f)\cup W)$.
\end{itemize}
In particular, there exists $n'\geq 0$ such that $f=f_0:U_\infty\to\EC\setminus\overline{V}_0$ is conjugate to $g:\phi_{n'+1}(U_\infty)\to\phi_{n'}(\EC\setminus\overline{V}_0)$ by the conformal map $\phi_{n'+1}|_{U_\infty}=\phi_{n'}|_{U_\infty}$.

Denote $c:=\phi_1(2)$ and $w_i:=\phi_0(z_i)$ for all $0\leq i\leq p-1$. Then $c$ and $w_0$ are critical points of $g$ in $\C$, where $c$ is attracted to $\infty$ and $\{w_0,w_1,\cdots,w_{p-1}\}$ is a superattracting cycle of $g$ with period $p$. Hence $J(g)$ is disconnected. Moreover, we know that  $w_1\in \phi_{n'}(U_1)$ and $w_0,w_2,\cdots, w_{p-1}$ in $\phi_{n'}(U_0)$.

\medskip
\textit{Step 4}. \textit{Cantor bubbles when $p\geq 2$}.
Let $K_1$ be the connected component of the filled Julia set of $g$ containing $w_1$. Then $g^{\circ p}(K_1) = K_1$ and $g^{\circ p}(J_1) = J_1$, where $J_1:=\partial K_1$.
By Theorem \ref{thm:McM1988}, there exist a rational map $g_0$ and a quasiconformal mapping $\phi: \widehat{\mathbb{C}} \to \widehat{\mathbb{C}}$ such that
\begin{equation}
\phi(J_1)=J(g_0) \text{\quad and\quad}\phi\circ g^{\circ p}=g_0\circ \phi \text{ on } J_1.
\end{equation}
Since the postcritical sets satisfy $P(g^{\circ p})=P(g)$ and $w_0,w_2,\cdots, w_{p-1}\in\phi_{n'}(U_0)$, we conclude that $g^{\circ p}$ contains exactly one critical value in $K_1$.

By a completely similar argument as Step 1 in the proof of Theorem \ref{thm:criterion}, we conclude that $J_1$ is a Jordan curve. This implies that $J_i$ is also a Jordan curve for all $0\leq i\leq p-1$, where $J_i=\partial K_i$ and $K_i$ is the connected component of the filled Julia set of $g$ containing $w_i$.
The remaining proof is the same as Theorem \ref{thm:criterion} (from Step 2 there) and $J(g)$ is a Cantor set with bubbles.
\end{proof}

\subsection{Specific examples II}

In this subsection, we apply Theorem \ref{thm:period} to construct a Cantor bubble cubic polynomial having a superattracting basin of period 2.
Moreover, we also construct a Cantor bubble quartic polynomial having two invariant but not completely invariant attracting or parabolic basins.

\begin{exam}[See Figure \ref{Fig:Julia-Cantor bubble-p}]\label{exam:bubble-p=2}
Consider the cubic polynomial
\begin{equation}
g_a(z)=z^3-\Big(a+\frac{1}{a}\Big)z^2+a, \text{\quad where }a\in\C\setminus\{0\}.
\end{equation}
If $a$ is large enough, then $J(g_a)$ is a Cantor set with bubbles.
If $a\neq 0$ is small enough, then the Julia components of $g_a$ contain countably many copies of the Julia set of $z\mapsto z^2-1$, which is not a Cantor set with bubbles.
\end{exam}

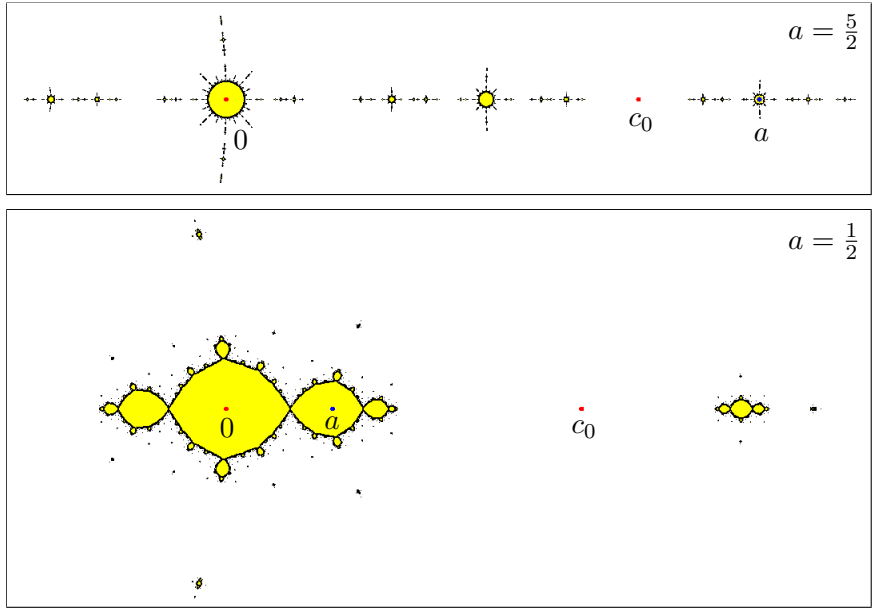
\begin{figure}[!htpb]
  \setlength{\unitlength}{1mm}
  \setlength{\fboxsep}{0pt}
  \centering
  \fbox{\includegraphics[width=0.8\textwidth]{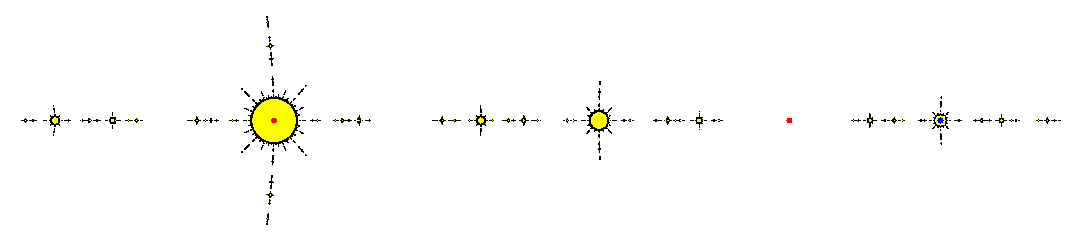}}
  \put(-84,6){$0$}
  \put(-15.5,7){$a$}
  \put(-11,20.5){$a=\frac{5}{2}$}
  \put(-32,9){$c_0$}\vskip0.15cm
  \fbox{\includegraphics[width=0.8\textwidth]{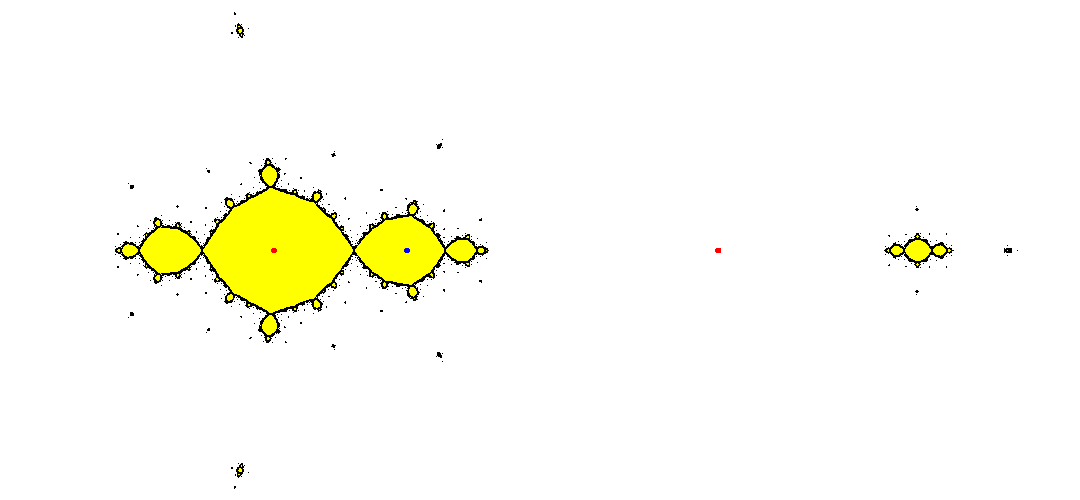}}
  \put(-85.8,22.5){$0$}
  \put(-72,23.6){$a$}
  \put(-11,47.5){$a=\frac{1}{2}$}
  \put(-39.5,23){$c_0$}
  \caption{Two different Julia sets of $g_a$ in Example \ref{exam:bubble-p=2}, where $a=\frac{5}{2}$ and $a=\frac{1}{2}$ are chosen such that $J(g_a)$ is, respectively, is not a Cantor set with bubbles (from top to bottom). Note that in both cases $g_a$ has a superattracting periodic cycle $\{0,a\}$ of period $2$.}
  \label{Fig:Julia-Cantor bubble-p}
\end{figure}

\begin{proof}
It is easy to verify that $g_a$ has a superattracting cycle $\{0,a\}$ of period two, and two critical points $0$ and $c_0=\frac{2}{3}(a+\frac{1}{a})$ in $\C$.
Denote $R_1:=(\frac{a}{10})^3$. If $a$ is large enough, then
\begin{enumerate}
\item[(a1)] $U_\infty:=\{z\in\EC:|z|>R_1\}$ is contained in the superattracting basin of $\infty$;
\item[(b1)] $|g_a(z)|>R_1$ for each $z\in \partial U_\infty\cup\{c_0\}$; and
\item[(c1)] $\overline{U}_0\cup\overline{U}_a\subset \EC\setminus \overline{U}_\infty$, where $U_0$ and $U_a$ are two different connected components of $g_a^{-1}(\EC\setminus \overline{U}_\infty)$ containing $0$ and $a$ respectively.
\end{enumerate}
We only verify the statement (c1). For large $a$, note that $a\in \EC\setminus \overline{U}_\infty$ has three preimages $0$, $0$ and $a+\frac{1}{a}$ in $\EC\setminus \overline{U}_\infty$ (counting with multiplicity). If $g_a^{-1}(\EC\setminus \overline{U}_\infty)$ is connected, then $c_0\in g_a^{-1}(\EC\setminus \overline{U}_\infty)\subset \EC\setminus \overline{U}_\infty$, which contradicts (b1). Thus $g_a^{-1}(\overline{U}_\infty)$ is contained in the superattracting basin of $\infty$, and $g_a^{-1}(\EC\setminus \overline{U}_\infty)$ consists of two different components $U_0$ and $U_a'$ which contain $0$ and $a+\frac{1}{a}$ respectively. Note that $a$ can be chosen large enough such that $U_a'=U_a$. By a completely similar proof to Theorem \ref{thm:period}, $J(g_a)$ is a Cantor set with bubbles.

\medskip
Denote $R_2:=(\frac{1}{10a})^3$. If $a\neq 0$ is small enough, then the following properties hold by similar calculations as above:
\begin{enumerate}
\item[(a2)] $U_\infty:=\{z\in\EC:|z|>R_2\}$ is contained in the superattracting basin of $\infty$;
\item[(b2)] $|g_a(z)|>R_2$ for each $z\in \partial U_\infty\cup\{c_0\}$; and
\item[(c2)] $\overline{U}_0\cup\overline{U_a'}\subset \EC\setminus \overline{U}_\infty$, where $U_0$ and $U_a'$ are two different connected components of $g_a^{-1}(\EC\setminus \overline{U}_\infty)$ such that $0, a\in U_0$ and $a+\frac{1}{a}\in U_a'$.
\end{enumerate}
The map $g_a:U_0\to \EC\setminus \overline{U}_\infty$ is a quadratic-like map (see \cite{DH85b} or \S\ref{subsec:poly-like}) with a superattracting cycle $\{0,a\}$. Hence the filled Julia  component of $g_a$ containing $0$ is homeomorphic to the filled Julia set of $z\mapsto z^2-1$. Therefore, $J(g_a)$ contains countably many components which are copies of $J(z^2-1)$ and it is not a Cantor set with bubbles.
\end{proof}

\begin{exam}[See Figure \ref{Fig:Julia-Cantor-bubble-attr-para}]\label{exam:bubble-attr-para}
Consider the quartic polynomial
\begin{equation}
h_a(z)=\frac{z^2(b+cz+dz^2)}{2a^2(a-1)(a-2)},
\end{equation}
where
\begin{equation}
b=-a(9a-8), \quad c=6a^2-4 \text{\quad and\quad} d=-4a+3
\end{equation}
with $a\in\C\setminus\{0,1,2,\frac{3}{4}\}$. Then $h_a$ has a superattracting fixed point $0$ and a parabolic fixed point $a$ with multiplier $1$.
If $a-1>0$ is small enough, then $J(h_a)$ is a Cantor set with bubbles.
\end{exam}

\begin{figure}[!htpb]
  \setlength{\unitlength}{1mm}
  \setlength{\fboxsep}{0pt}
  \centering
  \fbox{\includegraphics[width=0.85\textwidth]{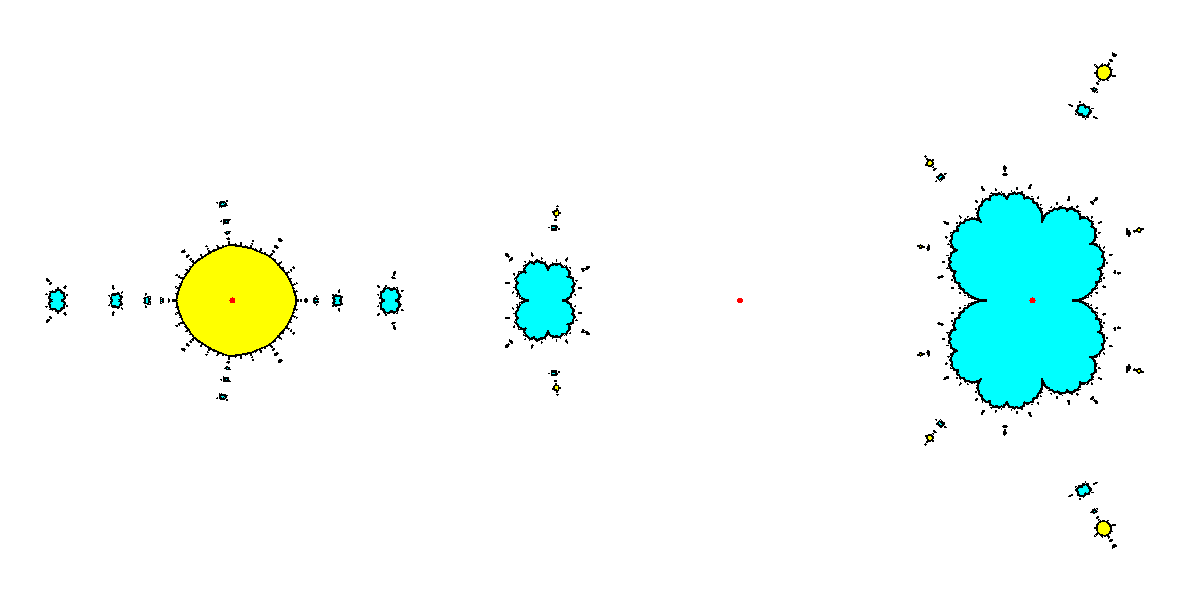}}
  \put(-98.3,26){$0$}
  \put(-18,26){$1$}
  \put(-14,27){$a$}
  \put(-47.5,26.5){$c_1$}
  \caption{A Cantor bubble Julia set of $h_a$ in Example \ref{exam:bubble-attr-para}, where $a=1.05$ is chosen such that $h_a$ has a superattracting fixed point $0$, a parabolic fixed point $a$, and three critical points $0$, $1$ and $c_1$.}
  \label{Fig:Julia-Cantor-bubble-attr-para}
\end{figure}
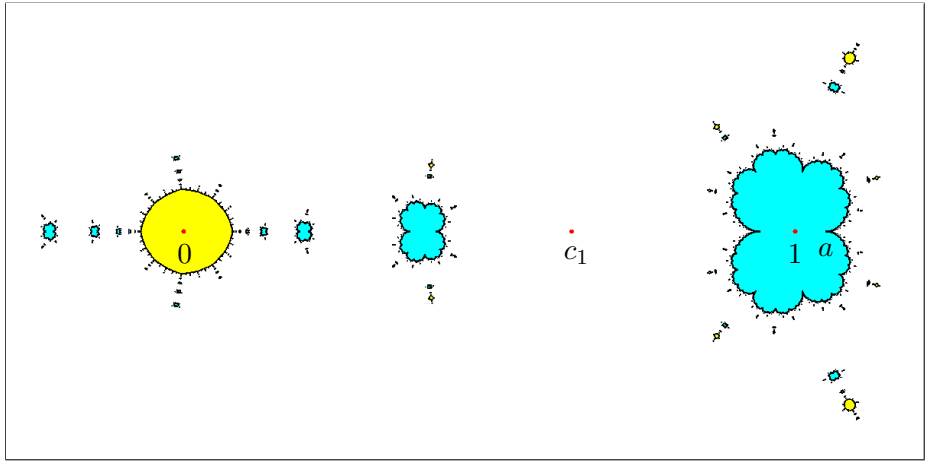

\begin{proof}
By direct calculations, $h_a$ has a superattracting fixed point $0$ and a parabolic fixed point $a$ with multiplier $1$, and critical points $0$, $1$ and $c_1=\frac{a(9a-8)}{2(4a-3)}$ in $\C$.
Moreover, if $a-1>0$ is sufficiently small, then
\begin{equation}
|h_a(c_1)|\asymp\frac{1}{a-1} \text{\quad and\quad} |h_a^{\circ 2}(c_1)|\asymp\frac{1}{(a-1)^5},
\end{equation}
and inductively, $c_1$ is contained in the superattracting basin of $\infty$. Thus the critical point $1$ is contained in the immediate parabolic basin of $a$ (see \cite[Theorem 10.15]{Mil06}).

Let $K_0$ and $K_1$ be the connected components of the filled Julia set $K(h_a)$ of $h_a$.
Note that $h_a$ is a real polynomial for real $a$. Every Fatou component of $h_a$ is symmetric about the real axis if it intersects with the real axis.
Since $c_1\in(0,1)$ is not contained in $K(h_a)$, it follows that $K_0\neq K_1$.
By a completely similar argument as Step 1 in the proof of Theorem \ref{thm:criterion}, we obtain that $\partial K_0$ and $\partial K_1$ are Jordan curves.
The rest proof is similar to that of Theorem \ref{thm:criterion} and $J(h_a)$ is a Cantor set with bubbles.
\end{proof}

\section{Dimensions of Cantor bubble Julia sets}\label{sec:dim}

The Cantor bubble Julia sets of rational maps obtained previously are geometrically finite and they have Hausdorff dimension strictly less than two. In this section we prove the existence of Cantor bubble Julia sets with Hausdorff dimension two.

\subsection{Polynomial-like maps and surgery}\label{subsec:poly-like}

Let $U$ and $V$ be two Jordan disks in $\C$ such that $U$ is compactly contained in $V$. The map $f:U\to V$ is called a \textit{polynomial-like map} of degree $d\geq 2$ if $f$ is a proper holomorphic surjection of degree $d$.
Polynomial-like maps of degree 2 and 3 are called \textit{quadratic-like} and \textit{cubic-like} maps respectively.
Let $K(f):=\bigcap_{n\geq 0}f^{-n}(V)$ be the \emph{filled Julia set} and $J(f):=\partial K(f)$ be the \textit{Julia set} of $f$.
Two polynomial-like maps $f_1:U_1\to V_1$ and $f_2:U_2\to V_2$ are said to be \emph{hybrid equivalent} if there is a quasiconformal mapping $h$ defined from a neighborhood of $K(f_1)$ onto a neighborhood of $K(f_2)$, which conjugates $f_1$ to $f_2$ and the complex dilatation of $h$ on $K(f_1)$ is zero. The following \textit{Straightening Theorem} is fundamental in the polynomial-like renormalization theory.

\begin{thm}[{\cite[p.\,296]{DH85b}}]\label{thm:straightening}
Let $f:U\to V$ be a polynomial-like map of degree $d\geq 2$. Then $f:U\to V$ is hybrid equivalent to a polynomial $P$ with the same degree $d$. Moreover, if $K(f)$ is connected, then $P$ is uniquely determined up to a conjugation by an affine map.
\end{thm}

A map $F:U\to\EC$ is called \textit{quasi-regular} if it can be written as $F=G\circ\phi$, where $\phi:U\to\phi(U)$ is a quasiconformal mapping defined in the open set $U$ and $G:\phi(U)\to\EC$ is holomorphic. For other equivalent characterizations of quasi-regular maps, see \cite[\S 1.6]{BF14a}.
Shishikura proved the following lemma, which is very useful when performing quasiconformal surgery. See also \cite[Proposition 5.2]{BF14a}.

\begin{lem}[{\cite[\S 3]{Shi87}}]\label{lema:qc}
Let $F:\EC\to\EC$ be a quasi-regular map. Suppose there exist an open set $E\subset\EC$ and an integer $N\geq 0$ satisfying the following two conditions:
\begin{itemize}
\item $F(E)\subset E$; and
\item $\frac{\partial F}{\partial \overline{z}}=0$ holds in $E$ and on $\EC\setminus F^{-N}(E)$ a.e.
\end{itemize}
Then there is a quasiconformal mapping $\phi:\EC\to\EC$ such that $\phi\circ F\circ \phi^{-1}$ is rational.
\end{lem}

\subsection{Proof of Theorem \ref{thm:dim-2}}

It is known that there exist cubic polynomials having Cantor Julia sets with Hausdorff dimension two \cite{Yan21}. We use quasiconformal surgery to combine such polynomials with quadratic polynomials to obtain a Cantor bubble Julia set with Hausdorff dimension two.

\begin{proof}[Proof of Theorem \ref{thm:dim-2}]
\textit{Step 1}. \textit{The map $f_0$}. Let $f_0$ be a cubic polynomial in \cite{Yan21} such that $J(f_0)$ is a Cantor Julia set having Hausdorff dimension two. We choose a large $R_0>1$ such that $f_0^{-1}(\D_{R_0})$ is a Jordan domain and compactly contained in $\D_{R_0}$, where $\D_r=\{z\in\C:|z|<r\}$ for $r>0$. Denote $U_0:=f_0^{-1}(\D_{R_0})$. Then $f_0: U_0\to\D_{R_0}$ is a cubic-like map.

\medskip
\textit{Step 2}. \textit{The map $f_1$}.
Let $z_1\in\D_{R_0}\setminus \overline{U}_0$ and $R_1>1$. We consider the following conformal map
\begin{equation}
\varphi_1(z):=R_1^2\,\frac{R_0(z-z_1)}{R_0^2-\overline{z}_1 z}:\D_{R_0}\to\D_{R_1^2}
\end{equation}
which maps $z_1$ to $0$.
Then $U_1:=\varphi_1^{-1}(\D_{R_1})$ is a disk containing $z_1$. We choose $R_1>1$ large enough such that $\overline{U}_1\subset \D_{R_0}\setminus \overline{U}_0$. Then
\begin{equation}
f_1:=\varphi_1^{-1}\big(\varphi_1(z)^2\big):U_1\to\D_{R_0}
\end{equation}
is a quadratic-like map and $z_1$ is a superattracting fixed point of $f_1$.

\medskip
\textit{Step 3}. \textit{The power map}.
Let $1<R<R_0$ be a number which is close to $R_0$ such that $\overline{U}_0\cup \overline{U}_1\subset\D_{R}$ and $R_0<R^5$. Then $\D_R\setminus(\overline{U}_0\cup \overline{U}_1)$ is a multiply connected domain having $3$ smooth boundary components.
Moreover, $\deg(f_0:\partial U_0\to\partial \D_{R_0})=3$, $\deg(f_1:\partial U_1\to\partial \D_{R_0})=2$ and $\deg(z\mapsto z^5:\partial \D_R\to\partial \D_{R^5})=5$. See Figure \ref{Fig-surgery}.

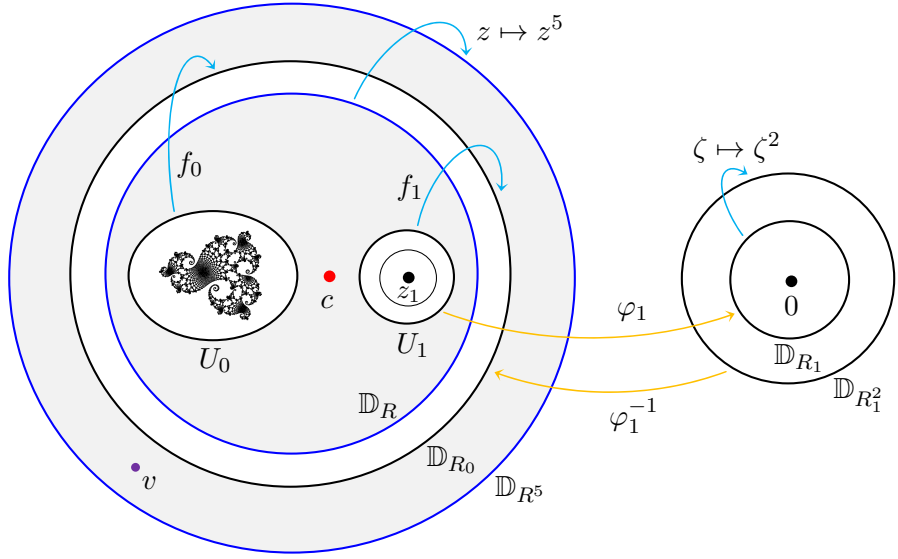
\begin{figure}[!htpb]
  \setlength{\unitlength}{1mm}
  \centering
  \includegraphics[width=0.85\textwidth]{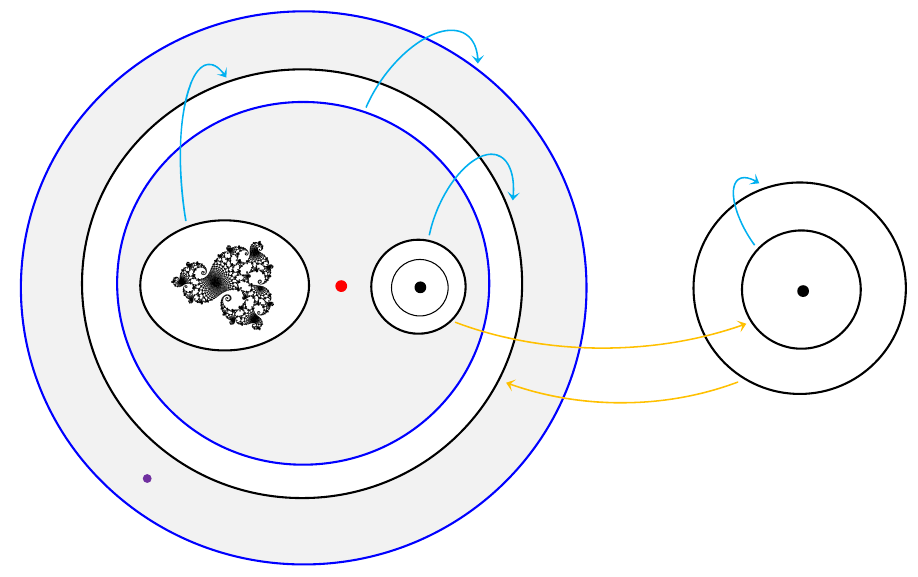}
  \put(-63,13){$\D_{R_0}$}
  \put(-72,20){$\D_R$}
  \put(-77,34){$c$}
  \put(-100.5,10){$v$}
  \put(-93,26){$U_0$}
  \put(-67,28){$U_1$}
  \put(-67,35.3){\small{$z_1$}}
  \put(-38,33){$\varphi_1$}
  \put(-39,18.5){$\varphi_1^{-1}$}
  \put(-54,9){$\D_{R^5}$}
  \put(-67,49){$f_1$}
  \put(-96,52){$f_0$}
  \put(-56.5,69.5){$z\mapsto z^5$}
  \put(-28,54){$\zeta\mapsto\zeta^2$}
  \put(-17.3,26.5){$\D_{R_1}$}
  \put(-9.5,22){$\D_{R_1^2}$}
  \put(-16,33){$0$}
  \caption{A sketch of the surgery construction of a degree five quasi-regular map which combines a cubic polynomial $f_0$ having a full dimensional Cantor Julia set and a quadratic-like map $f_1$ with a Jordan curve Julia set. Some special curves and points are marked.}
  \label{Fig-surgery}
\end{figure}


\medskip
\textit{Step 4}. \textit{The quasi-regular map}.
By Riemann-Hurwitz's formula and quasi-regular interpolation\footnote{One can paste a holomorphic ``model" of degree $5$ polynomial (which has an escaping critical point and two superattracting fixed points of local degrees $3$ and $2$ respectively) near the escaping critical point and three quasiconformal interpolations in annuli (for quasiconformal interpolation between annuli, see \cite[\S 2.3.2]{BF14a}) to obtain the desired quasi-regular map $g:\D_R\setminus(\overline{U}_0\cup \overline{U}_1)\to \D_{R^5}\setminus \overline{\D}_{R_0}$.}, there exists a continuous map
\begin{equation}
g:\overline{\D}_R\setminus(U_0\cup U_1)\to \overline{\D}_{R^5}\setminus \D_{R_0}
\end{equation}
such that
\begin{itemize}
\item $g:\D_R\setminus(\overline{U}_0\cup \overline{U}_1)\to \D_{R^5}\setminus \overline{\D}_{R_0}$ is a quasi-regular map of degree $5$ with exactly one critical point $c$ and one critical value $v$; and
\item $g|_{\partial U_0}=f_0$, $g|_{\partial U_1}=f_1$ and $g(z)|_{\partial \D_R}=z^5$.
\end{itemize}
Define
\begin{equation}\label{equ-F-attr-to-sup}
F(z):=
\left\{
\begin{array}{ll}
f_0(z)  &~~~~~~~\text{if}~z\in U_0, \\
f_1(z)  &~~~~~~~\text{if}~z\in U_1, \\
g(z) &~~~~~~\text{if}~z\in\overline{\D}_R\setminus(U_0\cup U_1), \\
z^5 &~~~~~~\text{otherwise}.
\end{array}
\right.
\end{equation}
Then $F:\EC\to\EC$ is a quasi-regular map of degree $5$.

Define $E:=\EC\setminus\overline{\D}_R$. Then $F(E)\subset E$ and $\overline{\D}_R\setminus(U_0\cup U_1)\subset F^{-1}(E)$. Hence $\frac{\partial F}{\partial \overline{z}}=0$ holds in $E\cup\big(\EC\setminus F^{-1}(E)\big)$ a.e. By Lemma \ref{lema:qc}, there exists a quasiconformal map $\phi:\EC\to\EC$ fixing $\infty$ such that
\begin{equation}
f:=\phi\circ F\circ \phi^{-1}
\end{equation}
is a polynomial of degree $5$. Moreover, $f:\phi(U_0)\to\phi(\D_{R_0})$ is a cubic-like map which is quasiconformally conjugate to $f_0:U_0\to\D_{R_0}$. By the choice of $f_0$, the filled Julia set $K_0:=\phi(K(f_0))$ of $f|_{\phi(U_0)}$ is a Cantor set having Hausdorff dimension two (see \cite[Corollary 13]{GV73}). Hence $J(f)$ has Hausdorff dimension two. It suffices to prove that $J(f)$ is a Cantor set with bubbles.

\medskip
\textit{Step 5}. \textit{$J(f)$ is a Cantor set with bubbles}.
Note that $\phi(\overline{\D}_R\setminus(U_0\cup U_1))$ is contained in the superattracting basin of $\infty$ and $f:\phi(U_1)\to\phi(\D_{R_0})$ is a quadratic-like map which is quasiconformally conjugate to $f_1:U_1\to\D_{R_0}$ and hence to $\zeta\mapsto \zeta^2$. Therefore, the boundary of the filled Julia component $K_1$ of $f$ containing the superattracting fixed point $\phi(z_1)$ is a quasicircle.

Let $K'$ be a component of the filled Julia set of $f$ which is not the preimage of $K_1$ nor of any component of $K_0$. If $K'$ is not a point, then there exists a periodic filled Julia component $K''$ of $f$ such that $K''$ contains a critical point (see \cite{QY09} or \cite{KS09}). However, the remaining critical point $\phi(c)$ is attracted by $\infty$, which is impossible. Therefore, $J(f)$ is a Cantor set with bubbles.
\end{proof}

By Lemma \ref{lem:attr-or-para}, Julia sets of quadratic rational maps cannot be Cantor bubbles, and all cubic polynomials having Cantor bubble Julia sets are geometrically finite and hence these Julia sets have Hausdorff dimension strictly less than two. We believe that there exist cubic rational maps having Cantor bubble Julia sets with Hausdorff dimension two.

\section{Uniformization of Cantor bubble Julia sets}\label{sec:quasis}

In this section, we give a sufficient condition such that some Cantor bubble Julia sets are quasisymmetrically equivalent to Cantor round bubbles.

\subsection{Bonk's criterion and distortion estimates}

A homeomorphism $f: X \to Y$ between two metric spaces $(X, d_X)$ and $(Y, d_Y)$ is \textit{quasisymmetric} if there exists a homeomorphism $\eta:[0, \infty) \to [0, \infty)$ such that
\[
\frac{d_{Y}(f(x), f(y))}{d_{Y}(f(x), f(z))} \leq \eta\left(\frac{d_{X}(x, y)}{d_{X}(x, z)}\right)
\]
for all distinct points $x, y, z \in X$. In this case, $X$ and $Y$ are called \textit{quasisymmetrically equivalent}.
It is known from \cite[Theorem 11.14]{Hei01} that an orientation preserving homeomorphism from $\EC$ to itself is quasisymmetric if and only if it is quasiconformal.

Let $\mathrm{dist}_{\EC}$ and $\mathrm{diam}_{\EC}$ denote the spherical distance and diameter respectively. The \emph{relative distance} of two disjoint subsets $A,B \subset \widehat{\mathbb{C}}$ is
\[
\Delta(A, B):=\frac{\mathrm{dist}_{\EC}(A, B)}{\min \left\{\mathrm{diam}_{\EC}(A), \mathrm{diam}_{\EC}(B)\right\}}.
\]
A collection of Jordan curves $\MS=\{\gamma_{i}\}_{i \in \mathbb{N}}$ is called \emph{uniformly relatively separated} if $\inf _{i \neq j} \Delta(\gamma_{i}, \gamma_{j})>0$, and \emph{uniform quasicircles} if there exists $K \geq 1$ such that each $\gamma_{i}$ is a $K$-quasicircle.
To study the quasisymmetric uniformization of Cantor bubble Julia sets, Bonk's following criterion is very useful.

\begin{thm}[{\cite[Theorem 1.1]{Bon11}}]\label{thm:Bonk}
Suppose $\mathcal{S}=\{ \gamma_i: i\in \mathbb{N}\}$ is a family of Jordan curves in $\EC$ bounding pairwise disjoint closed Jordan domains. If $\mathcal{S}$ consists of uniformly relatively separated uniform quasicircles, then there exists a quasiconformal map $\phi:\widehat{\mathbb{C}}\to\widehat{\mathbb{C}}$ such that $\phi(\gamma_i)$ is a round circle for all $i\in \N$.
\end{thm}

To obtain the uniformly relatively separated property, we use the \textit{conformal modulus} (or \textit{modulus} in short) of annuli to control the relative distances.

\begin{lem}[{\cite[Lemma 2.1]{QYZ19}}]\label{lem:QYZ}
Let $A \subset \EC$ be an annulus with two boundary components $\gamma_1$ and $\gamma_2$. If the modulus of $A$ satisfies $\operatorname{mod}(A) \geq m > 0$, then there exists a constant $C(m) > 0$ depending only on $m$ such that $\Delta(\gamma_1, \gamma_2) \geq C(m) > 0$.
\end{lem}

To obtain the property of uniform quasicircles, we use the following distortion lemma, whose proof is based on Koebe's distortion theorem.
One can also obtain the same result by using \cite[Lemma 6.1]{QWY12}.

\begin{lem}[{\cite[Lemma 2.8]{QYY18}}]\label{lem:uni-quasi}
Let $U_i \Subset V_i \neq \mathbb{C}$ be a pair of Jordan disks, where $i = 1,2$. Suppose $\operatorname{mod}(V_2 \setminus \overline{U}_2) \geq m > 0$ and $f:V_1 \to V_2$ is a
conformal map with $f(U_1) = U_2$. If $\partial U_2$ is a $K$-quasicircle, then there is a constant $C(K,m) > 1$ depending only on $K$ and $m$ such that $\partial U_1$ is a $C(K,m)$-quasicircle.
\end{lem}

\subsection{Proof of Theorem \ref{thm:roundbubbles}}

The \emph{$\omega$-limit set} $\omega(z)$ of a point $z\in \widehat{\mathbb{C}}$ under the rational map $f$ is the set of accumulation points
of the forward orbit of $z$, i.e.,
\[
\omega(z):=\big\{z'\in \widehat{\mathbb{C}}: \exists\,\{k_{n}\}_{n \in \mathbb{N}} \text{ such that } \lim _{n \to \infty} f^{\circ k_{n}}(z)=z'\big\}.
\]
If $f$ is a rational map having Cantor bubble Julia set $J(f)$, we denote
$$
\mathcal{\MU}(f):=\big\{\text{all simply connected Fatou components of $f$}\big\}
$$
and let $\MS(f):=\{\partial U: U\in\MU(f)\}$. Hence $\MS(f)$ consists of countably many disjoint Jordan curves.

\begin{lem}[{Uniformly relatively separated}]\label{lem:separated}
Let $f$ be a rational map having a Cantor bubble Julia set. If the boundary of every simply connected periodic Fatou component is disjoint from the $\omega$-limit sets of the critical points of $f$ in $J(f)$, then the curves in $\mathcal{S}(f)$ are uniformly relatively separated.
\end{lem}

\begin{proof}
We only give a sketch of the proof here since the idea is similar to \cite[Proposition 3.6]{QYZ19}.
Let $\mathcal{P}=\{P_1,\cdots,P_n\}$ be the collection of all simply connected periodic Fatou components of $f$. Iterating $f$ several times if necessary, we assume that the period of each $P_i$ is one. For $1\leq i\leq n$, let $Q_i$ be a simply connected domain containing $P_i$ such that $Q_1,\cdots,Q_n$ are mutually disjoint and each annulus $A_i:=Q_i\setminus \overline{P}_i$ contains no points in the critical orbits. Denote $m:=\min_{1\leq i\leq n}\textup{mod}(A_i)>0$.

Note that $f$ has only finitely many critical points. By the choice of $Q_i$, there exists a number $D\geq 1$ depending only on $f$ such that for any different $U_1,U_2\in\MU(f)$, there exist two minimal integers $n_1,n_2\geq 0$, and two simply connected domains $V_1$ and $V_2$ surrounding $U_1$ and $U_2$ respectively such that for $i=1,2$,
\begin{itemize}
\item $f^{\circ n_i}(U_i)=P_{k_i}\in\MP$ and $f^{\circ n_i}(V_i)=Q_{k_i}$ for some $1\leq k_1,k_2\leq n$; and
\item $f^{\circ n_i}:V_i\setminus \overline{U}_i\to A_{k_i}=Q_{k_i}\setminus \overline{P}_{k_i}$ is a covering map between annuli of degree at most $D$.
\end{itemize}
Thus we have $\textup{mod}(A)\geq m/D$, where $A:=\EC\setminus(\overline{U}_1\cup \overline{U}_2)$.
By Lemma \ref{lem:QYZ}, $\overline{U}_1$ and $\overline{U}_2$ are relatively separated with the relative distance $\Delta(\partial U_1,\partial U_2)$ depending only on $m$ and $D$.
Hence the curves in $\mathcal{S}(f)$ are uniformly relatively separated.
\end{proof}

\begin{defi}[Bounded degree condition]
Let $f$ be a rational map having an attracting periodic Fatou component $U$. We say that $f$ satisfies the \emph{bounded degree condition} on $\partial U$ if there exists $D\geq 1$ such that for any $z \in \partial U$, there is a number $\varepsilon_z>0$ such that for any integer $n \geq 0$ and any component $V_n(z)$ of $f^{-n}(B(z, \varepsilon_z))$ intersecting with $\partial U$, $V_n(z)$ is simply connected and $\deg(f^{\circ n}:V_n(z) \to B(z, \varepsilon_z))\leq D$.
\end{defi}

Here $B(z, \varepsilon_z)$ is the spherical disk centered at $z$ with radius $\varepsilon_z$.
We need the following result to obtain that each curve in $\MS(f)$ is a quasicircle.

\begin{lem}[{\cite[Proposition 6.1]{QWY12}}]\label{lem:BD-condition}
Let $f$ be a rational map having a simply connected attracting periodic Fatou component $U$. If $f$ satisfies the bounded degree condition on $\partial U$, then $\partial U$ is locally connected. If further $\partial U$ is a Jordan curve, then $\partial U$ is a quasicircle.
\end{lem}

We shall use the criterion in Lemma \ref{lem:BD-condition} to prove the following result.

\begin{lem}[{Uniform quasicircles}]\label{uniform_quasicircle-1}
Let $f$ be a rational map having a Cantor bubble Julia set. If the boundary of every simply connected periodic Fatou component is disjoint from the $\omega$-limit sets of the critical points of $f$, then the curves in $\mathcal{S}(f)$ are uniform quasicircles.
\end{lem}

\begin{proof}
Let $U\in\MU(f)$ be a periodic Fatou component. Since $\partial U$ is disjoint with the $\omega$-limit sets of all critical points of $f$, we conclude that $f$ has no parabolic periodic points on $\partial U$. By \cite[Theorem II]{Man93} and \cite[Lemma 2.7]{QYZ19}, $f$ satisfies the bounded degree condition on $\partial U$. By Lemma \ref{lem:BD-condition}, $\partial U$ is a quasicircle since $\partial U$ is a Jordan curve. Hence the curves in $\mathcal{S}(f)$ are quasicircles and it sufficient to prove that they are uniform quasicircles. We use a similar idea to \cite[Proposition 3.4]{QYZ19}.

Let $\mathcal{P}'=\{P_1,\cdots,P_\ell\}$ be the collection of all Fatou components in $\MU(f)$ such that each of them intersects with the critical orbits of $f$. Then there exists a constant $K'>1$ such that $\partial P_i$ is a $K'$-quasicircle for all $1\leq i\leq \ell$. For each $i$, let $Q_i$ be a Jordan domain containing $P_i$ such that $Q_1,\cdots,Q_\ell$ are mutually disjoint and each annulus $A_i:=Q_i\setminus \overline{P}_i$ contains no points in the critical orbits. Denote $m:=\min_{1\leq i\leq \ell}\textup{mod}(A_i)>0$.

By the definition of $\MP'$ and the choice of $Q_i$, for any Fatou component $U\in \MU(f)\setminus \MP'$, there exists a minimal integer $n\geq 1$ and a simply connected domain $V$ such that
\begin{itemize}
\item $f^{\circ n}(U)=P_{k}\in\MP'$ and $f^{\circ n}(V)=Q_{k}$ for some $1\leq k\leq \ell$; and
\item $f^{\circ n}:V\to Q_k$ is conformal.
\end{itemize}
By Lemma \ref{lem:uni-quasi}, the boundary $\partial U$ is a $K$-quasicircle, where $K=C(K',m)$ is a constant depending only on $K'$ and $m$.
By the arbitrariness of $U$, this implies that the curves in $\mathcal{S}(f)$ are uniform quasicircles.
\end{proof}

\begin{proof}[Proof of Theorem \ref{thm:roundbubbles}]
By Lemmas \ref{lem:separated} and \ref{uniform_quasicircle-1}, the curves in $\mathcal{S}(f)$ are uniform quasicircles and uniformly relatively separated.
According to Theorem \ref{thm:Bonk}, there exists a quasiconformal map $\phi: \widehat{\mathbb{C}} \to \widehat{\mathbb{C}}$ such that $\phi(\partial U)$ is a round circle for all $U \in \mathcal{U}(f)$. Hence $J(f)$ is quasisymmetrically equivalent to a Cantor set with round bubbles.
\end{proof}

In Example \ref{exam:bubble-para}, the cubic rational map $f_a$ has a superattracting fixed point $\infty$ and a parabolic fixed point $1$ with multiplier $1$.
If $-1-a>0$ is small enough, then $J(f_a)$ is a Cantor set with bubbles. Note that the boundary of every simply connected periodic Fatou component of $f_a$ is disjoint from the $\omega$-limit sets of the critical points. Thus $J(f_a)$ is quasisymmetrically equivalent to a Cantor set with round bubbles.

\bibliographystyle{amsalpha}
\bibliography{E:/Latex-model/Ref1}

\providecommand{\bysame}{\leavevmode\hbox to3em{\hrulefill}\thinspace}
\providecommand{\MR}{\relax\ifhmode\unskip\space\fi MR }
\providecommand{\MRhref}[2]{%
  \href{http://www.ams.org/mathscinet-getitem?mr=#1}{#2}
}
\providecommand{\href}[2]{#2}
\begin{thebibliography}{QWY12}

\bibitem[BF14]{BF14a}
B.~Branner and N.~Fagella, \emph{Quasiconformal surgery in holomorphic
  dynamics}, Cambridge Studies in Advanced Mathematics, vol. 141, Cambridge
  University Press, Cambridge, 2014.

\bibitem[BFH92]{BFH92}
B.~Bielefeld, Y.~Fisher, and J.~Hubbard, \emph{The classification of critically
  preperiodic polynomials as dynamical systems}, J. Amer. Math. Soc. \textbf{5}
  (1992), no.~4, 721--762.

\bibitem[BH92]{BH92}
B.~Branner and J.~H. Hubbard, \emph{The iteration of cubic polynomials. {II}.
  {P}atterns and parapatterns}, Acta Math. \textbf{169} (1992), no.~3-4,
  229--325.

\bibitem[BKM10]{BKM10}
A.~Bonifant, J.~Kiwi, and J.~Milnor, \emph{Cubic polynomial maps with periodic
  critical orbit. {II}. {E}scape regions}, Conform. Geom. Dyn. \textbf{14}
  (2010), 68--112.

\bibitem[BLM16]{BLM16}
M.~Bonk, M.~Lyubich, and S.~Merenkov, \emph{Quasisymmetries of {S}ierpi\'{n}ski
  carpet {J}ulia sets}, Adv. Math. \textbf{301} (2016), 383--422.

\bibitem[Bon06]{Bon06}
M.~Bonk, \emph{Quasiconformal geometry of fractals}, {ICM}. {V}ol. {II}, Eur.
  Math. Soc., Z\"{u}rich, 2006, pp.~1349--1373.

\bibitem[Bon11]{Bon11}
\bysame, \emph{Uniformization of {S}ierpi\'{n}ski carpets in the plane},
  Invent. Math. \textbf{186} (2011), no.~3, 559--665.

\bibitem[CG93]{CG93}
L.~Carleson and T.~W. Gamelin, \emph{Complex dynamics}, Universitext: Tracts in
  Mathematics, Springer-Verlag, New York, 1993.

\bibitem[CT11]{CT11}
G.~Cui and L.~Tan, \emph{A characterization of hyperbolic rational maps},
  Invent. Math. \textbf{183} (2011), no.~3, 451--516.

\bibitem[DG08]{DG08}
R.~L. Devaney and A.~Garijo, \emph{Julia sets converging to the unit disk},
  Proc. Amer. Math. Soc. \textbf{136} (2008), no.~3, 981--988.

\bibitem[DH85]{DH85b}
A.~Douady and J.~H. Hubbard, \emph{On the dynamics of polynomial-like
  mappings}, Ann. Sci. \'{E}cole Norm. Sup. (4) \textbf{18} (1985), no.~2,
  287--343.

\bibitem[DH93]{DH93}
\bysame, \emph{A proof of {T}hurston's topological characterization of rational
  functions}, Acta Math. \textbf{171} (1993), no.~2, 263--297.

\bibitem[DLU05]{DLU05}
R.~L. Devaney, D.~M. Look, and D.~Uminsky, \emph{The escape trichotomy for
  singularly perturbed rational maps}, Indiana Univ. Math. J. \textbf{54}
  (2005), no.~6, 1621--1634.

\bibitem[DM08]{DM08}
R.~L. Devaney and S.~M. Marotta, \emph{Evolution of the {M}c{M}ullen domain for
  singularly perturbed rational maps}, Topology Proc. \textbf{32} (2008),
  301--320.

\bibitem[DP09]{DP09}
R.~L. Devaney and K.~M. Pilgrim, \emph{Dynamic classification of escape time
  {S}ierpi\'{n}ski curve {J}ulia sets}, Fund. Math. \textbf{202} (2009), no.~2,
  181--198.

\bibitem[DRS07]{DRS07}
R.~L. Devaney, M.~M. Rocha, and S.~Siegmund, \emph{Rational maps with
  generalized {S}ierpinski gasket {J}ulia sets}, Topology Appl. \textbf{154}
  (2007), no.~1, 11--27.

\bibitem[DS97]{DS97}
G.~David and S.~Semmes, \emph{Fractured fractals and broken dreams}, Oxford
  Lecture Series in Mathematics and its Applications, vol.~7, The Clarendon
  Press, Oxford University Press, New York, 1997.

\bibitem[Etk22]{Etk22}
A.~Etkin, \emph{Dynamics of cubic rational maps under certain constraints on
  critical points}, Thesis (Ph.D.)--City University of New York, 2022.

\bibitem[GM10]{GM10}
A.~Garijo and S.~M. Marotta, \emph{Singular perturbations of {$z^n$} with a
  pole on the unit circle}, J. Difference Equ. Appl. \textbf{16} (2010),
  no.~5-6, 573--595.

\bibitem[GV73]{GV73}
F.~W. Gehring and J.~V\"{a}is\"{a}l\"{a}, \emph{Hausdorff dimension and
  quasiconformal mappings}, J. London Math. Soc. (2) \textbf{6} (1973),
  504--512.

\bibitem[HE25]{HE25}
J.~Hu and A.~Etkin, \emph{Cubic rational maps with escaping critical points,
  {P}art {II}: classification of {J}ulia sets in the case of a parabolic fixed
  point}, Discrete Contin. Dyn. Syst. \textbf{45} (2025), no.~11, 4427--4453.

\bibitem[Hei01]{Hei01}
J.~Heinonen, \emph{Lectures on analysis on metric spaces}, Universitext,
  Springer-Verlag, New York, 2001.

\bibitem[Hub16]{Hub16}
J.~H. Hubbard, \emph{Teichm\"{u}ller theory and applications to geometry,
  topology, and dynamics. {V}ol. 2, \textup{Surface homeomorphisms and rational
  functions}}, Matrix Editions, Ithaca, NY, 2016.

\bibitem[HXY25]{HXY25}
X.~He, Y.~Xiao, and F.~Yang, \emph{Rational maps whose {J}ulia sets are
  generalized {S}ierpi\'{n}ski gaskets}, Indiana Univ. Math. J., to appear,
  arXiv: 2503.18264, 2025.

\bibitem[Kle06]{Kle06}
B.~Kleiner, \emph{The asymptotic geometry of negatively curved spaces:
  uniformization, geometrization and rigidity}, {ICM}. {V}ol. {II}, Eur. Math.
  Soc., Z\"{u}rich, 2006, pp.~743--768.

\bibitem[KS09]{KS09}
O.~Kozlovski and S.~van Strien, \emph{Local connectivity and quasi-conformal
  rigidity of non-renormalizable polynomials}, Proc. Lond. Math. Soc. (3)
  \textbf{99} (2009), no.~2, 275--296.

\bibitem[LN24]{LN24}
Y.~Luo and D.~Ntalampekos, \emph{Uniformization of gasket {J}ulia sets}, arXiv:
  2411.17227, 2024.

\bibitem[LZ25]{LZ25}
Y.~Luo and Y.~Zhang, \emph{On quasiconformal non-equivalence of gasket {J}ulia
  sets and limit sets}, Ergodic Theory Dynam. Systems \textbf{45} (2025),
  no.~11, 3465--3489.

\bibitem[Ma{\~{n}}93]{Man93}
R.~Ma{\~{n}}{\'{e}}, \emph{On a theorem of {F}atou}, Bol. Soc. Brasil. Mat.
  (N.S.) \textbf{24} (1993), no.~1, 1--11.

\bibitem[Mar08]{Mar08b}
S.~M. Marotta, \emph{Singular perturbations of {$z^n$} with multiple poles},
  Internat. J. Bifur. Chaos Appl. Sci. Engrg. \textbf{18} (2008), no.~4,
  1085--1100.

\bibitem[McM88]{McM88}
C.~T. McMullen, \emph{Automorphisms of rational maps}, Holomorphic functions
  and moduli, {V}ol. {I} ({B}erkeley, {CA}, 1986), Math. Sci. Res. Inst. Publ.,
  vol.~10, Springer, New York, 1988, pp.~31--60.

\bibitem[MD08]{MD08}
M.~Morabito and R.~L. Devaney, \emph{Limiting {J}ulia sets for singularly
  perturbed rational maps}, Internat. J. Bifur. Chaos Appl. Sci. Engrg.
  \textbf{18} (2008), no.~10, 3175--3181.

\bibitem[Mil93]{Mil93}
J.~Milnor, \emph{Geometry and dynamics of quadratic rational maps, \emph{with
  an appendix by the author and L. Tan}}, Experiment. Math. \textbf{2} (1993),
  no.~1, 37--83.

\bibitem[Mil06]{Mil06}
\bysame, \emph{Dynamics in one complex variable}, third ed., Annals of
  Mathematics Studies, vol. 160, Princeton University Press, Princeton, NJ,
  2006.

\bibitem[Mil09]{Mil09}
\bysame, \emph{Cubic polynomial maps with periodic critical orbit. {I}},
  Complex dynamics, A K Peters, Wellesley, MA, 2009, pp.~333--411.

\bibitem[Nta26]{Nta26}
D.~Ntalampekos, \emph{Uniformization problems in the plane: a survey}, arXiv:
  2603.15098, 2026.

\bibitem[PT00]{PT00}
K.~M. Pilgrim and L.~Tan, \emph{Rational maps with disconnected {J}ulia set,
  \textup{G\'{e}om\'{e}trie complexe et syst\`emes dynamiques (Orsay, 1995)}},
  Ast\'{e}risque (2000), no.~261, 349--384.

\bibitem[PZ21]{PZ21}
F.~Przytycki and A.~Zdunik, \emph{On {H}ausdorff dimension of polynomial not
  totally disconnected {J}ulia sets}, Bull. Lond. Math. Soc. \textbf{53}
  (2021), no.~6, 1674--1691.

\bibitem[QWY12]{QWY12}
W.~Qiu, X.~Wang, and Y.~Yin, \emph{Dynamics of {M}c{M}ullen maps}, Adv. Math.
  \textbf{229} (2012), no.~4, 2525--2577.

\bibitem[QY09]{QY09}
W.~Qiu and Y.~Yin, \emph{Proof of the {B}ranner-{H}ubbard conjecture on
  {C}antor {J}ulia sets}, Sci. China Ser. A \textbf{52} (2009), no.~1, 45--65.

\bibitem[QY21]{QY21}
W.~Qiu and F.~Yang, \emph{Quasisymmetric uniformization and {H}ausdorff
  dimensions of {C}antor circle {J}ulia sets}, Trans. Amer. Math. Soc.
  \textbf{374} (2021), no.~7, 5191--5223.

\bibitem[QYY15]{QYY15}
W.~Qiu, F.~Yang, and Y.~Yin, \emph{Rational maps whose {J}ulia sets are
  {C}antor circles}, Ergodic Theory Dynam. Systems \textbf{35} (2015), no.~2,
  499--529.

\bibitem[QYY18]{QYY18}
\bysame, \emph{Quasisymmetric geometry of the {J}ulia sets of {M}c{M}ullen
  maps}, Sci. China Math. \textbf{61} (2018), no.~12, 2283--2298.

\bibitem[QYZ19]{QYZ19}
W.~Qiu, F.~Yang, and J.~Zeng, \emph{Quasisymmetric geometry of {S}ierpi\'{n}ski
  carpet {J}ulia sets}, Fund. Math. \textbf{244} (2019), no.~1, 73--107.

\bibitem[Roe06]{Roe06a}
P.~Roesch, \emph{On capture zones for the family
  {$f_\lambda(z)=z^2+\lambda/z^2$}}, Dynamics on the {R}iemann sphere, Eur.
  Math. Soc., Z\"{u}rich, 2006, pp.~121--129.

\bibitem[Shi87]{Shi87}
M.~Shishikura, \emph{On the quasiconformal surgery of rational functions}, Ann.
  Sci. \'{E}cole Norm. Sup. (4) \textbf{20} (1987), no.~1, 1--29.

\bibitem[Shi98]{Shi98}
\bysame, \emph{The {H}ausdorff dimension of the boundary of the {M}andelbrot
  set and {J}ulia sets}, Ann. of Math. (2) \textbf{147} (1998), no.~2,
  225--267.

\bibitem[Ste97]{Ste97}
N.~Steinmetz, \emph{Jordan and {J}ulia}, Math. Ann. \textbf{307} (1997), no.~3,
  531--541.

\bibitem[Ste06]{Ste06b}
\bysame, \emph{Sierpi\'{n}ski curve {J}ulia sets of rational maps}, Comput.
  Methods Funct. Theory \textbf{6} (2006), no.~2, 317--327.

\bibitem[Wan21]{Wan21}
X.~Wang, \emph{Hyperbolic components and cubic polynomials}, Adv. Math.
  \textbf{379} (2021), No. 107554, 42 pp.

\bibitem[XQY14]{XQY14}
Y.~Xiao, W.~Qiu, and Y.~Yin, \emph{On the dynamics of generalized {M}c{M}ullen
  maps}, Ergodic Theory Dynam. Systems \textbf{34} (2014), no.~6, 2093--2112.

\bibitem[Yan21]{Yan21}
F.~Yang, \emph{Cantor {J}ulia sets with {H}ausdorff dimension two}, Int. Math.
  Res. Not. (2021), no.~7, 4994--5006.

\bibitem[Yin92]{Yin92}
Y.~Yin, \emph{On the {J}ulia sets of quadratic rational maps}, Complex
  Variables Theory Appl. \textbf{18} (1992), no.~3-4, 141--147.

\bibitem[Zha10]{Zha10}
Y.~Zhai, \emph{A generalized version of {B}ranner-{H}ubbard conjecture for
  rational functions}, Acta Math. Sin. (Engl. Ser.) \textbf{26} (2010), no.~11,
  2199--2208.

\bibitem[ZJ09]{ZJ09}
G.~Zhang and Y.~Jiang, \emph{Combinatorial characterization of sub-hyperbolic
  rational maps}, Adv. Math. \textbf{221} (2009), no.~6, 1990--2018.

\end{thebibliography}

\end{document}